\documentclass[a4paper,11pt]{amsart}
\usepackage[foot]{amsaddr}
\usepackage{comment}
\usepackage{amscd}
\usepackage{amsmath}
\usepackage{amssymb}
\usepackage{epsf,latexsym,graphicx,dsfont}
\usepackage[matrix,arrow,tips,curve]{xy}

\usepackage{color}
\usepackage{tikz-cd}
\newlength{\defbaselineskip} 
\usepackage{footnote}
\usepackage{color}
\usepackage{todonotes}
\usepackage{MnSymbol}

\usepackage{amsfonts,amssymb,amscd}
\newtheorem{thm}{Theorem}[section]
\newtheorem{cor}[thm]{Corollary}
\newtheorem{corr}[thm]{Corollary}
\newtheorem{lemma}[thm]{Lemma}
\newtheorem{lem}[thm]{Lemma}
\newtheorem{prop}[thm]{Proposition}

\theoremstyle{definition}

\newtheorem{example}[thm]{Example}

\newtheorem{rem}[thm]{Remark}

\usepackage{fullpage}
\usepackage{tikz,tikz-cd}
\usetikzlibrary{matrix,arrows}
\makeatletter
\tikzset{
  edge node/.code={%
      \expandafter\def\expandafter\tikz@tonodes\expandafter{\tikz@tonodes #1}}}
\makeatother
\tikzset{
  subseteq/.style={
    draw=none,
    edge node={node [sloped, allow upside down, auto=false]{$\subseteq$}}},
  Subseteq/.style={
    draw=none,
    every to/.append style={
      edge node={node [sloped, allow upside down, auto=false]{$\subseteq$}}}
  }
}

 \numberwithin{equation}{section}
\numberwithin{equation}{section} \theoremstyle{definition}

\DeclareMathOperator{\Hom}{Hom}

\DeclareMathOperator{\Sym}{Sym}

 \DeclareMathOperator{\coker}{coker}

          \newcommand\PP{{\mathbb{P}}}

          \newcommand\cc{{\mathcal{L}  }}

           \newcommand\G{{\mathcal G}}

          \newcommand\oo{\mathcal O}
          \newcommand\Z{\mathbb{Z}}

          \newcommand\rk{\mathrm{rk}}

\definecolor{zielony}{rgb}{0.5, 0.9, 0.1}
\definecolor{czerwony}{rgb}{0.8, 0.2, 0.1}
\definecolor{niebieski}{rgb}{0.3, 0.1, 0.9}

\newcounter{appendice}

\begin{document}
\title{Symmetric locally free resolutions  \\
and rationality problems}
\author{Gilberto Bini}
\address[G.~Bini]{Universit\`a degli Studi di Palermo \\
Dipartimento di Matematica e Informatica \\
Via Archirafi 34 \\
90123 Palermo (Italy) \\
E-mail: {\tt gilberto.bini@unipa.it}
}
\author{Grzegorz Kapustka} 
\address[G.~Kapustka]{Jagiellonian University \\
Department of Mathematics and Informatics \\
\L ojasiewicza 6 \\ 
30-348 Krak\'ow (Poland)\\
E-mail: {\tt grzegorz.kapustka@uj.edu.pl}}
\author{Micha\l\ Kapustka}
\address[M.~Kapustka]{Institute of Mathematics of the Polish Academy of Sciences \\
ul. \'Sniadeckich 8 \\
00-656 Warszawa (Poland)\\
E-mail: {\tt michal.kapustka@impan.pl}
}

\begin{abstract}
We show that the birationality class of a quadric surface bundle over $\mathbb{P}^2$ is determined by its associated cokernel sheaves. As an application, we discuss stable-rationality of very general quadric bundles over $\PP^2$ with discriminant curves of fixed degree. In particular, we construct explicit models of these bundles for some discriminant data. Among others, we obtain various birational models of a nodal Gushel--Mukai fourfold, as well as of a cubic fourfold containing a plane. Finally, we prove stable irrationality of several types of quadric surface bundles.
%Finally, besides deducing the non stable-rationality of several types of quadric surface bundles, we prove that of some complete intersections of small degree.
\end{abstract}
\maketitle

\section{Introduction}\label{introduction}
A quadric bundle is a smooth variety $Q$ equipped with a flat morphism ${Q}\rightarrow B$ whose general fiber is a smooth quadric. The problem of rationality of quadric bundles has been a classical subject of investigation dating back to works of Artin and Mumford \cite{AM}, and considerable progress has been made recently: see, for instance, \cite{HPT} and \cite{Sch}. In this paper, we investigate quadric bundles of relative dimension 2 over the projective plane. We call them {\em quadric surface bundles}.

For a quadric bundle $Q$ of relative dimension $n$ over $\mathbb P^2$, there exists a vector bundle $\mathcal G$ of rank $n+2$ on $\mathbb P^2$ such that ${Q}$ is a divisor of relative degree $2$ on $\mathbb{P}(\mathcal G)$, see \cite[Proposition 1.2]{B3}. 
The quadric bundle $Q$ is then determined by a non-zero section $G\in H^0(\Sym^2(\mathcal G^{\vee})(\delta))$ for some $ \delta\in \mathbb Z$ or, equivalently, by a symmetric map 
$q_G: \mathcal G (-\delta) \to \mathcal{G}^{\vee}$ which induces the following short exact sequence:
\begin{equation}\label{lfres}
0\to\mathcal{G}(-\delta) \xrightarrow{q_G} \mathcal{G}^{\vee}\to \cc \to 0.
\end{equation}
In this case, $\cc$ is a sheaf supported on a curve defined by $\det q_G$, the so-called {\em discriminant curve}, and is determined by ${Q}$ up to a line bundle twist. We will call the {\em cokernel of $ Q$} the equivalence class $[\mathcal{L}]$ of the sheaf $\mathcal L$  with respect to the equivalence relation: $$\mathcal L_1\sim \mathcal L_2 \Leftrightarrow \exists \, \, \gamma \in \mathbb Z: \mathcal L_1= \mathcal L_2(\gamma).$$ The sheaves representing the cokernel of a quadric bundle $Q$ will be called {\em cokernel sheaves} associated with $Q$.

In this paper, we prove that the birational type of a quadric surface bundle  over ${\mathbb P}^2$ with smooth discriminant curve is determined by its cokernel.

\begin{thm}\label{main}  %Let $\mathcal{L}$ be a line bundle on a smooth plane curve. %Assume there exist two symmetric resolutions:
%\[0\to\mathcal{G}(-\delta_1)\xrightarrow{q_G} \mathcal{G}^{\vee}\to \cc(k_1)\to 0, \qquad \qquad 0\to\mathcal{F}(-\delta_2)\xrightarrow{q_F} \mathcal{F}^{\vee}\to \cc(k_2)\to 0,\]
%for some  vector bundles $\mathcal{G}$, $\mathcal{F}$ of rank 4 and $\delta_1,\delta_2,k_1,k_2 \in \mathbb{Z}$.
Let $ Q$ and ${Q}'$ be two quadric surface bundles over $\mathbb{P}^2$ having the same cokernel $[\mathcal{L}]$ supported on a smooth discriminant curve. {\color{black} Then the generic fiber of $ Q$ and that of $ {Q}'$ are isomorphic.}
\end{thm}
%Our proof of Theorem \ref{main} will be at the end of Section \ref{thm}.
We prove Theorem \ref{main} in two steps. First, in Theorem \ref{general IOOV}, we provide an explicit relation between the Brauer class associated with the quadric surface bundle and its cokernel. Second, we show in \S\ref{proof of main}, that two quadric surface bundles {\color{black} are isomorphic at the generic point of $\mathbb{P}^2$} if their corresponding Brauer classes are equal. In fact, the second step is valid also for quadric bundles defined over an open subset of $\mathbb{P}^2$, see Corollary \ref{birational quadric bundle}. As observed by A. Kuznetsov \cite{Kuz1}, one can alternatively prove Theorem \ref{main} by showing that any two quadric bundles with the same cokernel are linked through quadric reductions via quadric bundles of higher ranks.

%In this paper we investigate resolutions of symmetric sheaves on a smooth plane curve $C$ via locally free sheaves on the ambient projective space. Remarkably, we manage to deduce some geometric properties on the birational geometry of quadric bundles having $C$ as discriminant locus, which is a rather intriguing subject of investigation dating back to works of Artin and Mumford \cite{AM}, still a very active research area: see, for few examples, \cite{HPT} and \cite{Sch}.

Theorem \ref{main} is in fact valid in wider generality. In particular, we do not need to assume that the quadric bundles are of relative dimension $2$. Here we formulate Theorem \ref{main} for quadric surface bundles, as this is the only unknown case to date. For the relative dimension $1$ the theorem was proven in \cite[Proposition 3.10]{Pro}. Moreover, for relative dimension greater than or equal to $3$, quadric bundles over $\mathbb{P}^2$ always admit an odd degree multi-section (see \cite{Lang}); hence they are rational over $\mathbb{P}^2$. 

In the second part of the paper, we apply Theorem \ref{main} to the birational study of quadric bundles having smooth discriminant curves.  %Our first observation in that part is that we can reconstruct quadric bundles starting from a cokernel. %i.e. a two torsion sheaf or a theta divisor on a smooth curve $C$.
We first observe that for a quadric surface bundle $ Q$ over $\mathbb{P}^2$ the discriminant curve $C$ is always of even degree. Furthermore, any cokernel sheaf $\cc$ associated with $Q$ is a so-called {\em quasi half-period} on $C$, i.e., $\mathcal{L}^2=\mathcal{O}_C(\gamma)$ for some $\gamma\in \mathbb Z$. In particular, one can choose a cokernel sheaf associated to $Q$ that is one of the following 
\begin{enumerate}
\item $\mathcal L$ is trivial, i.e., $\mathcal{L}= \mathcal{O}_C$;
\item  $\mathcal L$ is a half-period, i.e., $\mathcal{L}^2=\mathcal{O}_C$ and $\mathcal{L}\neq \mathcal{O}_C$; %Such $ \mathcal L$ is called a half-period.
\item $ \mathcal L$ is an even theta characteristic, i.e.,  $\mathcal{L}^2=\omega_C$ and $h^0(\cc)$ is even; %Such $ \mathcal L$ is called an even theta characteristic.
\item $ \mathcal L$ is an odd theta characteristic, i.e., $\mathcal{L}^2=\omega_C$ and $h^0(\cc)$ is odd.
%Such $ \mathcal L$ is called an odd theta characteristic.
\end{enumerate}
If we fix the (even) degree of the discriminant curve, these cases define four irreducible (cf. \cite[Propositions 2 and 3]{B2} and \cite{V1}) moduli spaces of pairs $(C,\cc)$ of plane curves of fixed degree equipped with a suitable quasi-half period.
%, where $\cc$ is either trivial, or a half period or a theta characteristic, even if $h^0(\cc)$ is even or odd if $h^0(\cc)$ is odd. 
We refer to these four moduli spaces as moduli spaces of cokernels in the fixed degree.
By Theorem \ref{main} and the discussion above, the disjoint union of moduli spaces of cokernels in degree $2d$ provides a parametrization of the set of {\color{black} isomorphism  types  of generic fibers of} quadric surface bundles with smooth discriminant curves of degree $2d$.  %Moreover, by abuse of terminology, we will refer to (very) general quadric surface bundles of given type as quadric surface bundle whose cokernel is (very) general in its corresponding moduli space.
 %Indeed, let  $\cc$ be a sheaf supported on an even degree smooth curve $C\subset \PP^2$ such that $\cc^2=\oo_{C}(n)$ for some $n \in \mathbb{Z}$. We will call such a sheaf a {\em quasi half-period} supported on $C$. 

In our first application of Theorem \ref{main}, we provide explicit birational models of quadric surface bundles corresponding to general elements of each of the four moduli spaces of cokernels supported on low degree curves.
\begin{prop}\label{qw} Let $C$ be a general smooth plane curve of degree $2d\leq 12$. Let $Q$ be a quadric surface bundle over $\mathbb{P}^2$ with discriminant curve $C$. Then $Q$ {\color{black} is isomorphic at the generic point of $\mathbb{P}^2$ with} one of the following:
\begin{enumerate}
\item a trivial quadric surface bundle,
\item a complete intersection of $d$ divisors of bidegree $(1,1)$ and one divisor of bidegree $(0,2)$  in $\mathbb{P}^2\times \mathbb{P}^{d+3}$,
\item a complete intersection  of $d-2$ divisors of bidegree $(1,1)$ and one divisor of bidegree $(1,2)$ in $\mathbb{P}^2\times \mathbb P^{d+1}$,
\item the residual component  of an intersection of $d-3$ quadrics and a cubic containing a $\mathbb{P}^{d-1}$ in $\mathbb{P}^{d+2}$, with (rational) fibration structure given by projection from $\mathbb{P}^{d-1}$.
\end{enumerate}
\noindent Conversely, a general curve $C$ appears as a discriminant curve of some quadric surface bundle in any of the four families above. 
\end{prop}
The proof is given in Section \ref{rat}.
The constructions in Proposition \ref{qw} are obtained as follows. As proved in \cite{B,C}, a quasi-half period $\cc$ on a smooth plane curve of even degree always admits a symmetric free resolution, i.e., a resolution like (\ref{lfres}) with $\mathcal G$ a split bundle. In particular, $\mathcal L$ is a cokernel sheaf of a quadric bundle being a divisor in the projectivization of a split bundle, possibly of high rank. Applying quadric reduction (cf.\cite{ABB}) to these bundles,
we obtain many symmetric locally free resolutions of the same cokernel sheaf.
We infer, in particular, that $\cc$ is the cokernel of many quadric surface bundles.
These are contained in $\mathbb{P}^3$-bundles which are not necessarily projectivizations of split bundles. We are able to explicitly find suitable quadric reductions for all the families of high rank quadric bundles constructed in \cite{B,C} for $2d \leq 12$. This leads, among others, to the constructions in Proposition \ref{qw}.
%; however their classification for a given quasi half-period $\mathcal L$  on a fixed curve $C$ remains an open problem. %The map $q_G$ in (\ref{lfres}) is naturally associated to a non-zero section $G\in H^0(\Sym^2(\mathcal{G}^{\vee})(\delta))$ and induces a quadric bundle $Q_G\subset \PP(\mathcal{G}^{\vee})$.
%A natural question is to determine when two such bundles over ${\mathbb P}^2$ are birational or not.  Our main result in this direction is the following.
%Our proof of Theorem \ref{main}, which will be at the end of Section \ref{thm}, gives an explicit relation between the bundle $\cc$ and the Brauer class associated  with  the quadric bundle given by any of its resolutions.  As observed by A. Kuznetsov \cite{Kuz1}, one can alternatively prove Theorem~\ref{main} by showing that any two resolutions are linked through quadric reductions via quadric bundles of higher ranks.

The same method can be applied to provide alternative birational models for quadric surface bundles from Proposition \ref{qw}. In particular, in Corollary \ref{deg6}, we find alternative models for Verra fourfolds ($2:1$ covers of $\PP^2\times \PP^2$ branched along a $(2,2)$ divisor) and for cubic fourfolds containing a plane, blown up along the plane.
An interesting consequence is a comparison of two families of fourfolds for which the rationality of a general element is still a wide open conjecture, namely Verra fourfolds and Gushel--Mukai fourfolds. The latter are fourfolds obtained as proper intersections of a cone over the Grassmannian $G(2,5)$ in its Pl\"ucker embedding with a seven-dimensional quadric in $\mathbb{P}^{10}$. We prove that Verra fourfolds are birational to singular one-nodal Gushel--Mukai fourfolds. %varieties obtained as intersections of a cone over the Grassmannian $G(2,5)$ in its Pl\"ucker embedding with seven-dimensional quadrics.
In particular, combining that with results of \cite{NS}, we deduce that the stable irrationality of a very general Verra fourfold would also imply the stable irrationality of a very general smooth Gushel--Mukai fourfold.

Finally, we use Theorem \ref{main} and the specialization method (cf. {\color{black}\cite{ V, CTP, HPT,Sch1}}) to prove stable irrationality of quadric surface bundles whose associated cokernel is very general in its moduli space in some degrees. 
\begin{corr}\label{cor irrationality}
\begin{enumerate}
\item A  quadric surface bundle on $\PP^2_{\mathbb C}$ with  cokernel that is very general in the moduli space of  half-periods on curves of degree $8\leq 2d \leq 18$ is not stably rational.
\item  A  quadric surface bundle on $\PP^2_{\mathbb C}$ with  cokernel that is very general in the moduli space of  theta characteristics on curves of degree $8\leq 2d \leq 14$ is not stably rational.
\item A  quadric surface bundle on $\PP^2_{\mathbb C}$ with  cokernel that is very general in the moduli space of   odd theta characteristics on curves of degree $9\leq 2d \leq 12$ is not stably rational.
\end{enumerate}
\end{corr}
In particular, we obtain stable irrationality of the very general element of the following families.
\begin{enumerate}
\item a complete intersection of $d$ divisors of  bidegree  $(1,1)$ and one divisor of bidegree $(0,2)$ in $\mathbb{P}_{\mathbb C}^2\times \mathbb{P}_{\mathbb C}^{d+3}$ for $4\leq d\leq 9$,
\item a complete intersection of type $d-2$ divisors of  bidegree  $(1,1)$ and one divisor of bidegree $(1,2)$ in $\mathbb{P}_{\mathbb C}^2\times \mathbb P_{\mathbb C}^{d+1}$, for $4 \leq d\leq 7$,
\item the residual component of an intersection of $d-3$ quadrics and a cubic containing a $\mathbb{P}_{\mathbb C}^{d-1}$ in $\mathbb{P}_{\mathbb C}^{d+2}$ for $4 \leq d\leq 6$.
\end{enumerate}
We produce more stably irrational quadric surface bundles by applying Corollary \ref{cor irrationality} to alternative models of quadric surface bundles listed in Section \ref{ex}.

Throughout we mainly work with an algebraically closed field $\mathbf k$ of characteristic different from 2.  Whenever we use other assumptions on the field, we state them explicitly.

\section{Quadric bundles from symmetric resolutions of the same sheaf}\label{thm}
By \cite[Proposition 1.2]{B3} any quadric bundle over $\mathbb{P}^2$ arises from the following construction. 
Let $\mathcal{G}^{\vee}$ be a locally free sheaf on $\PP^2$. We use the notation $\mathbb{P}(\mathcal G)=\operatorname{Proj} S^{\bullet}(\mathcal G)$ for the projectivization by hyperplanes in the fibers. Denote by $\oo_{\PP(\mathcal G^{\vee})}(1)$ the tautological sheaf and let $0\neq G\in H^0(\Sym^2(\mathcal G^{\vee})(\delta)).$ Thus $G$ induces a family of quadric forms $ \oo_{\PP^2}\to \Sym^2(\mathcal{G}^{\vee})(\delta)$
and  a symmetric map $q_G\colon \mathcal{G}(-\delta)\to \mathcal{G}^{\vee}$  whose cokernel is a sheaf  $\cc$. Hence we have an exact sequence
\begin{equation}
\label{shortexactsequence}
0\to \mathcal{G}(-\delta)\xrightarrow{q_G} \mathcal{G}^{\vee}	\to \cc \to 0.
\end{equation}
Assume $q_ G$ is generically non-degenerate. Then the zero locus $Q_G\subset \PP(\mathcal{G}^{\vee})$ of $G$, equipped with the  natural projection map to $\mathbb{P}^2$, is a quadric bundle.

Set $C=\operatorname{Supp}(\cc)$, the scheme theoretic support of $\cc$.  Then $C$ is a plane curve and  $\deg C = 2c_1\left({\mathcal G}^{\vee}\right)+\rk({\mathcal G}^{\vee})\delta$. Observe that the curve $C$ is the  discriminant curve of the quadric bundle $Q_G$ i.e. the locus over which the fiber is singular. Indeed, this is exactly the locus where the symmetric map $q_G$ has lower rank. 

Note that not only the discriminant curve is determined by the quadric bundle $Q_G$. In fact, from the proof of \cite[Proposition 1.2]{B3} it follows that for every quadric bundle $Q$ over $\mathbb{P}^2$ there exists a unique projective bundle $\mathbf P$ over $ \mathbb{P}^2$ such that $Q$ appears on $\mathbf P$ as a divisor of relative degree 2. In other words, in our notation, we have $\mathbb P(\mathcal G)$ is determined by $Q_G$.  In particular, $\mathcal G$ is determined by $Q_G$ up to a twist by a line bundle on $\mathbb{P}^2$ and consequently also $\mathcal L$ is determined by $Q_G$ up to a twist by a line bundle on $\mathbb{P}^2$.

{\color{black} In what follows, we assume that all considered quadric bundles arise} from sheaves $\mathcal G$ of even rank and maps $q_G$ of generic maximal rank such that the discriminant curve $C$ is smooth and the corank two locus is empty. 

{\begin{lem}\label{quasi half period} If the corank two locus of $q_G$ is empty, then the sheaf $\cc$ is a quasi half-period on $C$, i.e., $\mathcal{L}^2=\mathcal{O}_C(\gamma)$ for some $\gamma\in \mathbb Z$..
\end{lem}
\begin{proof}
Here we closely follow Theorem 3.1 in \cite[p. 8]{DK}. The determinant of the short exact sequence \eqref{shortexactsequence} gives
a self-duality $ \cc\otimes \det(\mathcal{G}^{\vee})\to (\cc\otimes \det(\mathcal{G}^{\vee})(\delta))^{\vee}$ on the complement of the corank two locus inside the corank one locus, namely $C$. Under our assumptions, the corank two locus is empty, so the self-duality above is defined on the whole of $C$. We infer $\cc^2= (\det \mathcal G)^2\otimes \mathcal O_C(-\delta)=\mathcal O_C( \rk (\mathcal G)\delta -\deg C-\delta )$ which implies that $\cc$ is a quasi half-period.
\end{proof}
}

 Let $S$  be the double cover of $\mathbb{P}^2$ branched along $C$. We will call it {\em the discriminant double cover} associated to $Q_G$. We define $\beta_{Q_G}\in \operatorname{Br}(S)$ to be the class of the even Clifford algebra associated with $Q_G$ viewed as an algebra over its center, whose underlying scheme is $S$ (cf. \cite[\S 1.6]{ABB}). This will be identified, {\color{black}in Section \ref{Step 1}}, with a pullback of a class in the Brauer group  $\operatorname{Br}(\mathbf k(\mathbb P^2))$. 
\begin{rem}

More geometrically, over the complex numbers the Brauer--Severi scheme associated with a quadric bundle $Q\to\PP^2$ of relative dimension $2$ is the relative Hilbert scheme of lines in the fibers
of $\pi\colon\operatorname{Hilb} Q\to \PP^2$. This is a $\PP^1$-fibration over the discriminant double cover $S$. In fact, a smooth quadric surface admits two pencils of lines and a corank one quadric surface admits one pencil. The gluing data of this $\PP^1$-fibration gives a Brauer class $\beta_{Q}\in H^2(\oo_{S}^{\ast})_2\simeq \operatorname{Br}(S)$ (see, for instance, \cite{Gr,GS, Yoshioka}). 
\end{rem}

In order to prove Theorem \ref{main} we make an intermediate step by first proving that  $[\mathcal L]$ determines the Brauer class $\beta_{Q_G}$.

%\begin{rem}
%In general, if $Q$ is a quadric bundle of even dimension over $\mathbb{P}^2$, it is still possible to define $\beta_Q$ because the even Clifford algebra associated with $Q$ may be viewed as an algebra over its center, which is supported on the discriminant double cover of $Q$ (cf. \cite[\S 1.6]{ABB}). 
%\end{rem}

\begin{thm}\label{general IOOV}
Let $C$ be a smooth curve of degree $2d$ and $\cc$ a quasi half-period on $C$ considered as a sheaf on $\mathbb{P}^2$. Let 
$$0 \to \mathcal G(-\delta_1) \xrightarrow{q_G} \mathcal G^{\vee}\to \cc(k_1)\to 0, \qquad \qquad 0 \to \mathcal F(-\delta_2)\xrightarrow{q_F} \mathcal F^{\vee} \to \cc(k_2)\to 0,$$
be locally free resolutions, where $\mathcal G$, $\mathcal F$ are vector bundles of even rank on $\mathbb{P}^2$, $\delta_1,\delta_2,k_1,k_2\in \mathbb Z $ and  $q_G$, $q_F$ are symmetric forms corresponding to $ G\in H^0(\Sym^2(\mathcal G^{\vee})(\delta_1))$ and $ F\in H^0(\Sym^2(\mathcal F^{\vee})(\delta_2))$.
Let $Q_G\subset \mathbb{P}(\mathcal G^{\vee})$ and $Q_F\subset\mathbb{P}(\mathcal F^{\vee})$ be the two quadric bundles associated with these resolutions. Then $\beta_{Q_G}=\beta_{Q_F}\in \operatorname{Br}(S)$.
\end{thm}
 
\subsection{The proof of Theorem \ref{general IOOV}} \label{brauer}
We formulate the proof so that we work only with the quadric bundle $Q_G$ and prove that its Brauer class $\beta_{Q_G}$ is determined by $[\mathcal L]$.  The proof {\color{black} will be divided into several steps.}
 \subsubsection{\bf Step 0. Reduce to the case $\delta_1=\delta_2$ and $k_1=k_2=0$}{
 The preparatory step is based on the trivial observation that twisting a resolution 
 $$0 \to \mathcal G(-\delta) \xrightarrow{q_G} \mathcal G^{\vee}\to \cc(k)\to 0,$$
 by $\mathcal O_{\mathbb P^2}(-k)$, we obtain a new resolution of $\cc$, namely:
  \begin{equation}\label{resolution G'} 0 \to \mathcal G'(-\delta') \xrightarrow{q_G} \mathcal G'^{\vee}\to \cc\to 0,\end{equation}
with $\mathcal G'=\mathcal G(k)$ and $\delta'=\delta+2k$. The latter resolution leads to the same quadric bundle. 
Finally, Lemma \ref{quasi half period} implies that the integer $\delta'$ appearing in a resolution of the form (\ref{resolution G'}) is uniquely determined by $C$, $\mathcal L$ and $\operatorname{rk}(\mathcal G)=\operatorname{rk}(\mathcal G')$. Indeed, 
$$\cc^2(\deg C)= \mathcal O_C( (\rk (\mathcal G')-1)\delta' ).$$
}

\subsubsection{\bf Step 1. Trivialization over an open set in $\PP^2$} \label{Step 1} {Choose a line $L\subset \mathbb{P}^2$ and set $U=\mathbb{P}^2\setminus L$. The map $q_G|_U$  induces a quadric $b$ over the field of rational functions $\mathbf k(\PP^2)$. Indeed, by the famous theorem of Quillen--Suslin $\mathcal G|_U$ is a split bundle over $U$ and the map $q_G|_U$ is given by a symmetric matrix with entries being regular functions on $U$, which are in particular rational functions in $\mathbf k(\PP^2)$.
 
Accordingly, $b$ defines a Clifford algebra $\operatorname{Cl}_U:=\operatorname{Cl}(q_{G}|_U)$. Denote by $\operatorname{Cl}_0^U$ the algebra of even parts, as described in \cite[\S.~3.2]{IK}.
We look at $\operatorname{Cl}_0^U$ as an algebra over its center $\mathbf k(S)$, where $\pi\colon S\to\PP ^2$ is the discriminant double cover. Hence we can associate to it an element $[\operatorname{Cl}_0^U]$ in $\operatorname{Br}(\mathbf k(S))$.
From \cite[V.2.4]{Lam}, the algebra $\pi|_U^* (\operatorname{Cl}_U)$ is Morita equivalent to $\operatorname{Cl}_0^U$. %\pi: \operatorname{Hilb}(Q_G) \to {\mathbb P}^2$. 
%Moreover, $\pi_U^* (\operatorname{Cl}_U)$ is an Azumaya algebra over $S$ (cf. \cite[\S 1.6]{ABB}).
We define $ \beta_{U}:=[\operatorname{Cl}_0^U]=\pi_U^*[\operatorname{Cl}_U]$.
What's more, by \cite[Theorem 4.5]{IK}, $\beta_{U}$ is uniquely determined by $[\operatorname{Cl}_U]$ and does not depend on $U$ because $\beta_{U}=[\operatorname{Cl}_0^U]\in \operatorname{Br}(\mathbf k(S))$ is equal to the class in $\operatorname{Br}(S)\subset \operatorname{Br}(\mathbf k(S))$ defined as in \cite[\S 1.6]{ABB}.
%of the Brauer--Severi scheme (\cite[\S 4]{AM}), which is obtained as the relative Hilbert scheme of lines on $Q$. 
Therefore, $\beta_U$ may also be denoted by $\beta_{Q_G}$. 

 Since we have a symmetric map between trivial bundles over $U$, which is given by a matrix of rational functions with denominators being powers of a fixed linear form, we can apply \cite[Theorem 6.1]{IOOV} to conclude that the Brauer class $\beta_U=\beta_{Q_G}$ is determined by $\mathcal L$. This ends the proof of Theorem \ref{general IOOV}. 
 
 For the sake of completeness, we will outline the proof of  \cite[Theorem 6.1]{IOOV} in our context by means of slightly modified arguments. 
\subsubsection{\bf Step 2.  The Brauer class is determined by its residue along $C$}\label{outline}
Since we assume that $\mathbf k$ is algebraically closed and $\operatorname{char}\mathbf k\neq 2$, the exact sequences in \cite[Theorem~1]{AM} {\color{black}and \cite[\S 3.2]{IOOV}}  give the following diagram:
\begin{equation}\label{qwer}\hspace{-0.3cm}
\xymatrix{
0 \ar[r] &\operatorname{Br}(\mathbf k(\mathbb P^2))[2]  \ar[r]^{\hspace{-0.9cm}\partial} \ar[d]^{\pi^*}& \displaystyle \bigoplus_{x\in (\mathbb P^2) ^{(1)}} H^1(\mathbf k(x),\mathbb{Z}/2)  \ar[r]^{\hspace{0.9cm}\rho} \ar[d]^{\pi^*} &  \displaystyle    \bigoplus_{P\in (\mathbb P^2)^{(2)}} {\mathbb{Z}/2}  \ar[d]^{\pi^*}\\
\operatorname{Br}(S)[2] \ar@{^{(}->}[r] &\operatorname{Br}(\mathbf k(S))[2]  \ar[r]^{\hspace{-0.9cm}\partial} & \displaystyle \bigoplus_{x\in (S) ^{(1)}} H^1(\mathbf k(x),\mathbb{Z}/2)  \ar[r]^{\hspace{0.9cm} \rho}&   \displaystyle   \bigoplus_{P\in (S)^{(2)}} {\mathbb{Z}/2} ,}
\end{equation}
where $T^{(r)}$ is the set of points of codimension $r$ of $T$; $\mathbf k(x)$ is the residue field of $x$; by  $V[2]$ we mean the $2$-torsion subgroup of the group $V$ and $H^1$ refers to Galois cohomology. Finally, the horizontal maps are induced by residues. In particular, the  map $\rho$ is illustrated in \cite[\S.~3.2, diagram 4.1, proof of Proposition 4.3]{IOOV}} (it is denoted by $r$ in the reference). 

\begin{comment}Actually, the exact sequence coming from \cite[Theorem~1]{AM} does not involve the $2$-torsion group; this can be however adapted as follows. We will refer to the top horizontal line in the diagram above, as the bottom horizontal line can be obtained similarly. First, recall that by definition we have
$$
\operatorname{Br}(\mathbf k(\mathbb P^2))[2] \simeq \operatorname{Br}(\mathbf k(\mathbb P^2))/2\operatorname{Br}(\mathbf k(\mathbb P^2));
$$
moreover, residues induce a map from $2\operatorname{Br}(\mathbf k(\mathbb P^2))$ to $\bigoplus_{x\in (\mathbb P^2)^{(1)}} H^1(\mathbf k(x),2\left(\mathbb{Q}/\mathbb{Z}\right))$; hence there is a map $$
\operatorname{Br}(\mathbf k(\mathbb P^2))[2] \to \bigoplus_{x\in (\mathbb P^2) ^{(1)}} H^1(\mathbf k(x),\mathbb{Z}/2).
$$
\end{comment}

This induces the following diagram over $U$:
 {\small
\begin{equation}\label{4.1 IOOV}\hspace{-0.5cm}
\xymatrix{
0 \ar[r] &\operatorname{Br}(U)[2]   \ar[r]^{\hspace{-1.2cm} \partial} \ar[d]^{\pi^*}& \displaystyle \bigoplus_{x\in (\mathbb P^2\setminus U) ^{(1)}} H^1(\mathbf k(x),\mathbb{Z}/2)  \ar[r]^{\hspace{0.4cm} \rho} \ar[d]^{\pi^*} &     \displaystyle \bigoplus_{P\in (\mathbb P^2\setminus U)^{(2)}} {\mathbb{Z}/2}  \ar[d]^{\pi^*}\\
\operatorname{Br}(S)[2] \ar@{^{(}->}[r]^{\hspace{-0.5cm} \iota} &\operatorname{Br}(\pi^{-1}(U))[2]  \ar[r]^{\hspace{-1 cm}  \partial} & \displaystyle  \bigoplus_{x\in (S\setminus \pi^{-1}U) ^{(1)}} H^1(\mathbf k(x),\mathbb{Z}/2)  \ar[r]^{\hspace{0.4cm} \rho }&  \displaystyle    \bigoplus_{P\in (S\setminus \pi^{-1}U)^{(2)}} {\mathbb{Z}/2} .}
\end{equation}}

To conclude that the Brauer class $\beta_U=\beta_{Q_G}$ depends only on $\mathcal L$, it suffices to prove that $\mathcal L$ determines the class $[\operatorname{Cl}_U]$ or, equivalently, its residues along all the curves in $\mathbb{P}^2$. Since the ramification locus of the quadric over $\mathbf k(\mathbb{P}^2)$ is contained  in $C\cup L$, only the residues along $C$ and $L$ matter. Moreover, it is not hard to see that the residue along $L$ is in fact determined by the residue along $C$.
Indeed, since $\operatorname{Cl}_U=\operatorname{Cl}(M)$ and the residue of $\operatorname{Cl}_U$ along the preimage of $L$ on $S$ is trivial (see \S \ref{brauer}), this residue is of the form $(\frac{f}{{\mathbf l}^{2d}})^i$, where $f$ is the equation of $C$ and ${\mathbf l}$ is the 
equation of a line (any line will give an equivalent residue) in $\PP^2$ and $i\in \mathbb{Z}$ (cf.~\cite[Proposition ~6.3]{IK}).
In order to find $i$ we use the description \cite[Theorem~1]{AM} of the map $\rho$ in (\ref{qwer}). At points $P\in C\cap L$ we have $\rho(h_C,h_L)=v_P(h_C)+v_P(h_L)$ where $h_C$ and $h_L$ are elements from $H^1(\mathbf k(C),\Z/2\Z)\simeq \mathbf k(C)^{\ast}/(\mathbf k(C)^{\ast})^2$ and $$H^1(\mathbf k(L),\Z/2\Z)\simeq \mathbf k(L)^{\ast}/(\mathbf k(L)^{\ast})^2$$ respectively. Hence, by (\ref{4.1 IOOV}) we deduce that the valuation of $\operatorname{Cl}(M)$ along $C\cap L$ determines the valuation along $L$. \footnote{Note that these valuations do not have to be equal. In fact, in the case when $\cc$ is a theta characteristic we shall see that $v_P(h_C)=1$
at a point $P\in L\cap C$ (see \cite[Proposition~6.3]{IOOV}).}

\subsubsection{ \bf Step 3. An explicit computation of the residue along $C$}\label{nres} Here we provide a  computation of the residue along $C$ which is based on a suggestion by an anonymous referee. In particular, we will prove that the residue of the class $\beta_U$ along $C$ is given by the restriction of the quadratic form $b$ to the curve $C$ or, equivalently, it is determined by $\mathcal L$.

Let $R\subset K=\mathbf k(\PP^2)$ be the local ring of the generic point of $C\subset \PP^2$.
The bilinear form $b \colon \G\times \G\to \oo(\delta)$ induced by $q_G$ gives a map of $K$-vector spaces $V\times V\to K$, where $V$ is the vector space over $K$ corresponding to $\G$. The latter map restricts to $R$-modules as $b_R\colon\ V\times V\to I$, where $I$ is the $R$-module corresponding to $\oo(\delta)$. Set $$V^{\#}=\{ v\in V\otimes_R K : b(v,V)\subset I\}.$$
By standard results, $b_R$ induces an embedding $V\to V^{\#}=V^{\vee}\otimes I$. Moreover, the restriction of $b_R$ to $V^{\#}/V$ is also a non-degenerate bilinear form on the torsion $R$-module $V^{\#}/V$.

Since $C$ is smooth and is defined scheme theoretically as the discriminant of $q_G$,  $V^{\#}/V$ is isomorphic to $R/\langle f\rangle $, where $f$ is a generator of the maximal ideal of $R$, i.e., a local equation of $C$. Moreover, $R/\langle f\rangle =\mathbf k(C)$ is the function field of $C$. By \cite[Chapter 6, Theorem 2.1]{Scha} there is an equivalence of categories
between non-degenerate bilinear forms over finite length torsion $R$-modules and non-degenerate bilinear forms over the residue field. By \cite[Chapter 6, Remark 2.5]{Scha}, the residue of $b_R$ along $C$ is the element $a\in \mathbf k(C)^{\ast}$ (defined up to $(\mathbf k(C)^{\ast})^2$) corresponding to the restricted bilinear form $b_C\colon \mathbf k(C)\to \mathbf k(C)$. Note that $b_C$  is determined by the quasi-half period $\mathcal L$ on $C$. Moreover, if we compare \cite[Chapter 6, Definition 2.5]{Scha} with \cite[Proposition 2.29]{IOOV}, we deduce that the residue of the quadratic form $b_R$ is equal to the residue of the  Clifford algebra $\operatorname{Cl}_U$ with $[\operatorname{Cl}_U]=\beta_U$. 

To this end, fix a symmetric $m\times m$ matrix $M$ of rational functions on $U$ which defines the restriction $q_G|_U$ and has non-zero principal minors. %Denote by $M$ the matrix obtained from $\bar M$ by clearing the denominators in each entry by means of multiplying it with a suitable power of the equation of the line $L$.
Let us denote by ${M}_{i}$ the determinant of the matrix obtained by taking the first $i$ rows and columns of $M$. Then we can assume that such a matrix has valuation $0$ along  $C$ for $i\leq m-1$ and  $\det M=\frac{f}{l^2d}$ has valuation $1$.
Hence, by \cite[Proposition 2.29]{IOOV} the residue $\partial_C(\operatorname{Cl}_U)=\partial_C(\operatorname{Cl}(M))$ can be described by the restriction of
$ M_{1\dots m-1}$ to the curve $C$, namely:
\begin{equation} \label{residue in C1}
\partial_C(\operatorname{Cl}(M_{m-1}))= {M}_{m-1}.%=\frac{M_{m-1}}{\mathbf l^{\deg M_{m-1}}}.
\end{equation}
On the other hand, in \cite[Chapter 6, Definition 2.5]{Scha} we have
\begin{equation} \label{residue in C2}
\partial_{2,C}(q_G)= \frac {l^{2d} M_m}{f{M}_{m-1}}.%=\frac{\mathbf l^{\deg M_{m-1}}}{M_{m-1}}.
\end{equation}
Therefore, $\partial_{2,C}(q_G)=(\partial_C(\operatorname{Cl}(M_{m-1})))^{-1}$; hence the two residues determine the same element in $\mathbf k(C)^*/(\mathbf k(C)^*)^2$.

\subsubsection{\bf Step 4. Conclusion: The Brauer class does not depend on the resolution of $\cc$} Summing up, $\beta_{Q_G}$ does not depend on the bundle $\mathcal{G}$ and the symmetric locally free resolution of $\mathcal{L}$, it only depends  on $[\cc]$. This finishes the proof of Theorem \ref{general IOOV}.\qed

\subsection{The proof of Theorem \ref{main}}\label{proof of main}
We already know from Theorem \ref{general IOOV} that the quadric bundles associated with two symmetric resolutions of sheaves related by a twist  induce the same Brauer class.
On the other hand, if two quadric surface bundles induce the same $2$-torsion Brauer class, the associated (as in Section \ref{brauer}) Azumaya algebras (over the same central simple algebra) of rank two on $S$ are Morita equivalent. As a consequence, they are isomorphic over an open set $U\subset S$ (see \cite[\S 6]{IK}). Theorem 1 in \cite{APS} implies that the associated quadric bundles are isomorphic over an open subset of $\PP^2$; hence {\color{black} their generic fibers over $\mathbb{P}^2$ are isomorphic.}
\qed

In fact, the following more precise statement relating the Brauer class and the quadric bundle, was suggested to the authors by one of the referees. 
\begin{prop} \label{birationality for quadrics} Let $\mathbf k$ be any field of characteristic different from 2.
Two smooth quadric surfaces $Q_1$ and $Q_2$ over $\mathbf k$ are $\mathbf k$-birational to each other if and only if one of the following holds:
\begin{itemize}
    \item[i)] $Q_1$ and $Q_2$ both have a rational point;
    \item[ii)] none of $Q_1$ and $Q_2$ has a rational point, $\operatorname{disc}(Q_1)=\operatorname{disc}(Q_2) =\Delta \in \mathbf k^*/(\mathbf k^*)^{2}$ and in 
$\operatorname{Br}(\mathbf k(\sqrt{\Delta}))$ the associated Brauer classes $\alpha_1$ and $\alpha_2$ coincide.
\end{itemize}

 Furthermore, if the discriminants and Brauer classes associated to $Q_1$ and $Q_2$ are equal then $Q_1$ and $Q_2$ are isomorphic.\end{prop}

\begin{proof} 
A smooth quadric $Q \subset \PP^3_{\mathbf k}$ is defined by a non-degenerate quadratic form $q$ in four variables. To $q$ we associate the discriminant $\Delta \in  \mathbf k^*/( \mathbf k^*)^{2}$ that gives a trivial or a quadratic extension $K={\mathbf k}(\sqrt{\Delta})$ of $\mathbf k$. The quadric $Q$ also defines a quaternion class $\alpha$ in the Brauer group $Br(K)$. By \cite[Theorem 2.5]{CS}, if $C/K$ is the conic associated to $\alpha$ we have two possibilities:
\begin{itemize}
\item $\Delta=1$ (hence $K=\mathbf k$) and $Q = C \times_{\mathbf k} C $
\item $\Delta \neq 1$ and $Q = \operatorname R_{K/\mathbf k}(C)$, where $\operatorname R_{K/\mathbf k}$ is the Weil restriction of scalars. 
\end{itemize}

Hence, if we have two quadrics $Q_1$, $Q_2$ with the same discriminant $\Delta_1=\Delta_2=\Delta \in \mathbf  k^*/(\mathbf k^*)^{2}$ and associated Brauer class $\alpha_1=\alpha_2=\alpha \in  \operatorname{Br}(\mathbf k(\sqrt{\Delta}))$, then $Q_1=Q_2$. This is a proof of the last assertion of the proposition, but it also gives the ``if'' part of the first assertion in the case both quadrics have no $\mathbf k$-rational points. 

To deal with the case when both quadrics have a $\mathbf k$-rational point, recall that a quadric $Q$ has a rational point ($Q(\mathbf k)=\emptyset$) if and only if it is $\mathbf k$-rational, i.e., $\mathbf k$-birational to projective space. Therefore, if both $Q_1$ and $Q_2$ have a $\mathbf k$-rational, each of them is ${\mathbf k}$-birational to $\PP^2_{\mathbf k}$; hence they are $\mathbf k$-birational to each other. This concludes the ``if" part of the proof. 

For the ``only if" part, let  $Q_1$ and $Q_2$ be two smooth quadric surfaces which are $\mathbf k$-birational to each other. Clearly, $Q_1$ has a $\mathbf k$-rational point if and only if $Q_2$ does. Now, assume we have two quadrics $Q_1$ and $Q_2$ such that the sets of rational points $Q_1(\mathbf k)$ and $Q_2(\mathbf k)$ are empty. %We will prove that in that case we also have $Q_1=Q_2$ which will imply in particular the ``only if'' part of the proposition. 

For $i=\{1,2\}$ denote by $F_i := \mathbf k(Q_i)$ the function field of $Q_i$. Let $\Delta_i$ be the discriminant o $Q_i$ and $K_i$ the corresponding quadratic extension. By \cite[Theorem 3.1]{ACP}, the kernel of the restriction map $\operatorname {Br}(\mathbf k)\to \operatorname{Br}(F_i)$ is zero, unless $Q_i(\mathbf k) = \emptyset$, and $\Delta = 1$. In this case, the kernel is isomorphic to $\mathbb{Z}_2$ and it is spanned by the class $\alpha_i \in \operatorname{Br}(\mathbf k)$, which is non-zero.

Since $Q_1$ and $Q_2$ are $\mathbf k$-birational to each other, we have
$$B:=\operatorname{ker}[ \operatorname{Br}(\mathbf k)\to \operatorname{Br}(F_1)] = \operatorname{ker}[ \operatorname{Br}(\mathbf k) \to \operatorname{Br}(F_2)].$$

If $B\neq 0$, then $\Delta_1 =1$ and $\Delta_2 =1$ and the group $B$ is spanned by the unique class $\alpha=\alpha_1=\alpha_2\in \operatorname{Br}(\mathbf k)$. %attached to a quaternion algebra. , defining a conic $C$ over $k$, then $Q_1=C\times_k C =Q_2$.

If $B=0$, then $\Delta_1\neq 1$ and $\Delta_2\neq 1$. Suppose $\Delta_1\neq \Delta_2 \in  \mathbf k^*/( \mathbf k^*)^{2}$. Let us work over to the field $K_1$. Since $Q_1(\mathbf k)=Q_2(\mathbf k)=\emptyset$ we have $Q_1(K_1) = \emptyset$ and   $Q_2(K_1) = \emptyset$.   By \cite[Theorem 3.1]{ACP}, the
kernel of $\operatorname{Br}(K_1) \to \operatorname{Br}(K_1(Q_1))$ is isomorphic to $\mathbb{Z}_2$, whereas the kernel of $\operatorname{Br}(K_1) \to \operatorname{Br}(K_1(Q_2))$ is trivial (since the discriminant $\Delta_{2,K_1}$ of $Q_2$ over $K_1$ is not  $1$); this is however a contradiction, as $Q_1$ and $Q_2$ are birational over $\mathbf k$ and, accordingly, also over $K_1$.

Next assume $B = 0$ and $\Delta_1 = \Delta_2 \in \mathbf  k^*/(\mathbf  k^*)^{2}$ and this element is not $1$. Let $K:=K_1=K_2$ be the corresponding 
quadratic extension of $\mathbf k$. Since $Q_1$ is $\mathbf k$-birational to $Q_2$, this also holds over $K$. As a consequence, we have $K(Q_1)=K(Q_2)$ and the corresponding discriminants of $Q_1$ and $Q_2$ over $K$ are $\Delta_K=1\in K^*/(K^*)^2$. Therefore, again by \cite[Theorem 3.1]{ACP}, we have $\alpha_1 = \alpha_2\in \operatorname{Br}(K)$, as they are both equal to the generator of $ \operatorname{ker} [\operatorname{Br}(K) \to \operatorname{Br}(K(Q_2))]$.
% Thus $C_1=C_2$ over $K$, and finally $Q_1=R_{K|k}(C_1) = R_{K|k}(C_2)=Q_2$.
\end{proof}

Proposition \ref{birationality for quadrics} {\color{black} allows us} to deal with rationality questions for quadric bundles over open subsets of $\mathbb{P}^2$. Indeed, to a quadric bundle $Q$ over an open subset of $\mathbb{P}^2$ we can always associate a quadric over the function field of $\mathbb{P}^2$, {\color{black} by base change over the generic point}. In this way $Q$ has a well defined discriminant, discriminant double cover and an associated Brauer (see \cite{APS} for more details). Using this terminology we can now formulate the following corollary which is a straightforward consequence of Proposition \ref{birationality for quadrics}.  

\begin{cor}\label{birational quadric bundle}  Let $\mathbf k$ be any field of characteristic different from 2. Let $Q_1$ and $Q_2$ be two quadric surface bundles over some open subset of $\mathbb{P}^2_{\mathbf k}$. Assume they induce the same discriminant double cover and corresponding Brauer class. Then $Q_1$ and $Q_2$ are $\mathbf k$-birational.
\end{cor}

\begin{rem} We also have an alternative more geometric argument to deduce Theorem \ref{main} from Theorem \ref{general IOOV}. It works when the base field is $\mathbb C$.
{\color{black}Indeed, recall from \cite[Prop.~7.1.6]{CS1} that 
% \label{R} 
two Severi-Brauer varieties over a field $\mathbf{k}$ of the same dimension are isomorphic if and only if they determine the same Brauer class.}

%two Brauer--Severi varieties of dimension $1$ over a field $\mathbf k$ that give the same element in $\operatorname{Br}(\mathbf k)$ are birational.

Now, in order to deduce Theorem \ref{main} from Theorem \ref{general IOOV}, let us consider the generic fibers of the associated Brauer--Severi schemes $\operatorname{Hilb} (Q_F)\to S \ \ \text{and}\ \ \operatorname{Hilb} (Q_G)\to S$. The fiber over the generic point $\mathbb{C}(S)$ of the Brauer--Severi scheme is a Brauer--Severi variety over $\mathbb{C}(S)$ (it becomes $\PP^n$ after tensoring by
the algebraic closure $\overline{\mathbb{C}(S)}$). {\color{black}Using  injectivity of the restriction to the generic fiber 
the Brauer classes associated with $F$ and $G$ in $\operatorname{Br}(S)[2]$ can be identified with their restrction to their generic fibers which are the Brauer classes in $\operatorname{Br}(\mathbb{C}(S))$ associated to the Brauer--Severi varieties} (see, for example, {\color{black}\cite[Corr.~1.10]{Gr3},} \cite[\S 3]{AM}, \cite{Ko}).
The corresponding Brauer--Severi varieties are hence {\color{black}isomorphic}. In that case  the quadric bundles have isomorphic generic fibers.
Indeed, {\color{black}in order to see it, it suffices to observe that a point on a quadric bundle} is determined by the intersection of two lines (from different rulings). 
\end{rem}

\section{Free symmetric resolutions and quadric reductions}\label{ex}
Our aim in this section is to provide tools for the construction of symmetric locally free resolutions, involving bundles of rank four, of a quasi half-period $\mathcal L$ on a  general curve  $C$ of fixed even degree $2d$. For these purposes, we
recall the operation of quadric reduction on a quadric bundle \cite{ABB}. Here we adapt it to our context.  

\subsubsection*{\bf Quadric reduction} Let ${Q}$ be a quadric bundle, which is associated with 
a symmetric map $\varphi \colon \mathcal{T}(-\delta)\to \mathcal{T}^{\vee}.$ Let $\mathcal{U}$ be a regular isotropic subbundle of $\mathcal{T}(-\delta)$, i.e., such that  for  the inclusion map $\iota\colon\mathcal{U}\to \mathcal{T}(-\delta)$ we have $\iota^{\vee} \circ \varphi \circ  \iota=0$ and $\mathbb{P}(\mathcal{U}^{\vee})$ does not meet the singular locus of any fiber of ${Q}$.
Define the subbundle $\mathcal{U}^{\perp}:= \ker( \iota^{\vee} \circ \varphi)$. Denote by ${Q}'$ the quadric fibration associated with $ \varphi'\colon \mathcal{U}^{\perp}/\mathcal{U} \to (\mathcal{U}^{\perp}/\mathcal{U})^{\vee}(\delta)$ which is obtained by restricting $\varphi$ to  $\mathcal{U}^{\perp}$ and by taking the well-defined quotient map, as ${\mathcal U}$ is an isotropic sub-bundle. Then ${Q}'\subset \mathbb{P}((\mathcal{U}^{\perp}/\mathcal{U})^{\vee})$ will be called {\em the quadric reduction} of ${Q}$ with respect to the projective subbundle $\mathbb{P}(\mathcal{U}^{\vee})$ (see \cite[Definition 1.13]{ABB}). 

%If $r_1$ and $r_2$ are the quadric reductions with respect to the bundles $\mathbb{P}(\mathcal{U}^{\vee})$ and $\mathbb{P}(\mathcal{U'}^{\vee})$, we get quadric bundles in the projective bundles constructed as $\mathbb{P}(({\mathcal U}_i^{\perp}/{\mathcal U}_i)^{\vee})$ for $i=1,2$.

\begin{rem} A special case of quadric reduction is given by sub-bundles ${\mathcal U}$ of rank one. Then the corresponding quadric bundle admits a section $s$ and ${ Q}'$ is associated to the quadratic form induced by ${Q}$ on the projectivization of the relative tangent bundle of ${ Q}$ along the section $s$.
\end{rem}

\begin{prop}[{\cite[Theorem~1.27]{ABB}}]\label{prop quadric reduction brauer class}  The quadric bundle ${Q}$ and its quadric reduction $Q_{\mathcal U}$ with respect to any regular isotropic subbundle $\mathcal U$ admit the same discriminant locus and induce the same Brauer class on the discriminant double cover.
\end{prop}
\begin{rem} Theorem \ref{general IOOV} provides a straightforward alternative proof of Proposition \ref{prop quadric reduction brauer class} when the discriminant curve is smooth and the quadric bundles have even relative dimension. 
\end{rem}

\begin{rem} \label{reduction to rank 4} For every quadric bundle of rank greater than $4$ over $\mathbb{P}^2$ there is always an isotropic section that is regular over a dense open subset of $\mathbb{P}^2$. The corresponding quadric reduction leads to a quadric bundle of lower rank over that open subset or a so-called weak quadric bundle (where we allow the quadratic form to be 0 on some fibers) over $\mathbb{P}^2$. Thus, having any quadric bundle $Q$ of even rank larger than 4  we can perform quadric reductions of $Q$ with respect to its isotropic section and iterate this process until we get a quadric surface bundle over some dense open subset of $\mathbb{P}^2$. These (weak) quadric surface bundles {\color{black} will all be isomorphic at the generic point of $\mathbb{P}^2$} by Corollary \ref{birational quadric bundle}. 
\end{rem}

We now have a tool to construct quadric bundles of lower dimension from quadric bundles of higher dimension without changing the associated cokernel $[\cc]$.  For the construction of symmetric locally free resolutions involving bundles of rank four, we will hence look for such resolutions in some even rank $\geq 4$. First, it is natural to investigate free resolutions.    
We  proceed to the classification of quadric bundles associated with free resolutions.
% This will be based on two propositions.
\begin{prop} \label{split resolution} Let $C$ be a general smooth plane curve and  $\cc$ a quasi half-period on it. Let 
$0\to \mathcal T(-\delta)\xrightarrow{\phi} \mathcal T^{\vee}\to \cc \to 0$
be a free symmetric resolution of  $\cc$ such that $\Hom(\mathcal T^{\vee},\mathcal T(-\delta))=0$. Then there exists a free symmetric resolution  $$0\to \mathcal A(-\delta)\xrightarrow{\psi} \mathcal A^{\vee}\to \cc \to 0$$ if and only if 
\begin{equation}\label{decomposition of A}
\mathcal A=\mathcal T\oplus \mathcal{O}(\frac{\delta }{2})^l \oplus  \bigoplus_{j=1}^r ({\mathcal O}(k_j)\oplus {\mathcal O}(\delta-k_j) )
\end{equation}
 for some non-negative integers $l, k_j\geq 0$ such that $l=0$ for $\delta$ odd. Moreover, if $l$ is even, the  associated quadric fibrations $Q_{\phi}$ and $Q_{\psi}$ are related by a quadric reduction. Conversely, a generic symmetric map $\widetilde \psi\colon \mathcal A(-\delta)\rightarrow \mathcal A^{\vee}$ induces a quadric fibration $Q_{\widetilde \psi}$ which admits a quadric reduction  to some quadric fibration $Q_{\widetilde \phi}$ associated with a map $\widetilde \phi\colon\mathcal T(-\delta)\to \mathcal T^{\vee}$.
 \end{prop}
\begin{proof}
Clearly, if $\mathcal A$ is as in  (\ref{decomposition of A}),  i.e. $\mathcal A=\mathcal T\oplus \mathcal A'$ for $$\mathcal A':=\mathcal{O}(\frac{\delta }{2})^l \oplus  \bigoplus_{j=1}^r ({\mathcal O}(k_j)\oplus {\mathcal O}(\delta-k_j) ),$$ then we can take $\psi=\phi\oplus \operatorname{id}$.

To prove the other implication, take the resolution $0\to \mathcal A(-\delta)\xrightarrow{\psi} \mathcal A^{\vee}\to \cc \to 0$ with $\mathcal A$ being a split bundle. We have a commutative diagram 
$$
\begin{tikzcd}
0 \ar[r] & \mathcal A(-\delta) \ar[r,"\psi"] \ar[d,shift left=0.5ex] & \mathcal A^{\vee} \ar[r] \ar[d, shift left=0.5ex, "\pi_{\mathcal T}"] & \cc \ar[r] \ar[d, equal]& 0\\
0  \ar[r] &  \mathcal T(-\delta)\ar[r,"\phi" ]\ar[u,shift left=0.5ex]& \ar[r]  \mathcal T^{\vee}\ar[u,shift left=0.5ex,"\iota_{\mathcal A}"]  \ar[r,"s"] & \cc \ar[r]& 0 
\end{tikzcd}
$$
Then the composition of $\mathrm{id}-\pi_{\mathcal T}\circ \iota_{\mathcal A}\colon\mathcal T^{\vee}\to \mathcal T^{\vee}$ with $s$ is trivial, as the diagram is commutative. Hence it induces a homomorphism  $ \mathcal T^{\vee}\to \mathcal T(-\delta)$, which is trivial under the assumption $\operatorname{Hom}(\mathcal T^{\vee}, \mathcal T(-\delta))=0.$ Therefore, $\mathcal T^{\vee}$ must be a component of $\mathcal A^{\vee}$. Hence $\mathcal A=\mathcal T\oplus \mathcal A'$ with $A'$ and $\pi_{\mathcal T}$ is the projection $\mathcal A^{\vee}\to \mathcal T^{\vee}$. Furthermore, the map
 $\pi_{\mathcal T} \circ \psi$  has the same rank as $\phi$ at every point. In particular, since $C$ is irreducible, all maximal minors of the matrix associated with $\pi_{\mathcal T}\circ \psi$ are divisible by $\det \phi$. Therefore, a base change in $\mathcal A$ allows us to assume that  the matrix associated to $\psi$ is a diagonal block matrix with diagonal blocks associated to elements of the matrices $\operatorname{Sym}^2(\mathcal T^{\vee})(\delta)\ \  \text{and} \ \ \operatorname{Sym}^2(\mathcal A'^{\vee})(\delta)$. Under this assumption,  the map $\psi$ induces an isomorphism between $\mathcal A'(-\delta)$ and $\mathcal A'^{\vee}$, which gives the assertion.  Indeed, for every line bundle component  $\mathcal O(-k_i)$ of $\mathcal A'(-\delta)$ (i.e.~a component $\mathcal O(\delta-k_i)$ of $\mathcal A'$) the isomorphism associates a component of $\mathcal A'^{\vee}$ isomorphic to it  and this correspondence is a bijection. Hence such a component of $\mathcal A'^{\vee}$ must come from a component of the form $\mathcal O(k_i)$ in $\mathcal A'$.
 To prove the assertion about the quadric reduction, it is enough to observe that another base change transforms $\psi$ into a matrix consisting of two diagonal blocks of size $l\times l$ and $2r\times 2r$; both of them can be put in anti-diagonal form with $1$ on the anti-diagonal. If $l$ is even, we can easily see a rank $l+r$ regular isotropic subbundle of $\mathcal A$, with respect to which the quadric reduction of $Q_{\psi}$ is $Q_{\phi}$. 
The same argument is valid for any map $\widetilde \psi \colon \mathcal A(-\delta)\rightarrow \mathcal A^{\vee}$.
\end{proof}

\subsection{\bf Symmetric free resolutions of quasi half-periods} \label{free resolutions}

{\color{black} Applying Proposition \ref{split resolution} to the minimal free resolutions of general quasi-half periods found in \cite{B} and \cite[Proposition 4.2 and Proposition 4.6 ]{C},} we obtain a classification of split symmetric resolutions of quasi half-periods on general smooth plane curves of degree $2d$. Indeed, we have the following characterization of split  bundles $\mathcal A$ involved in symmetric resolutions of quasi half-periods $\cc$,  namely 
\begin{equation}\label{resolution A} 0\to \mathcal A(-\delta)\to \mathcal A^{\vee}\to \cc\to 0.\end{equation}
Note first that, if $\cc$ is a quasi-period on a general curve $C$, there exists $\gamma\in \mathbb Z$ such that $\cc(\gamma)$ is trivial, a half-period, an even or an odd theta characteristic. Twisting the resolution \ref{resolution A} by $\mathcal O_{\mathbb P^2}(\gamma)$ we obtain 
\begin{equation}\label{twisted resolution A} 0\to \mathcal A_{\gamma}(-\delta_{\gamma})\to \mathcal A_{\gamma}^{\vee}\to \cc(\gamma)\to 0,\end{equation}
with $\mathcal A_{\gamma}=\mathcal A(-\gamma)$ and $\delta_{\gamma}=\delta-2\gamma$. Hence we reduced our classification of symmetric free resolutions of non trivial quasi-half periods to the three cases listed below.

\begin{description}
\item[Half-periods] If $\cc(-d)$ is a half-period then $\delta=0$ and we have
 \begin{equation}\label{eq half period}\mathcal A=\mathcal{O}(-1)^d\oplus  \mathcal{O}^l \oplus  \bigoplus_{j=1}^r ({\mathcal O}(k_j)\oplus {\mathcal O}(-k_j) ) \end{equation} for some non-negative integers $l, k_j\geq 0$.

\item[Even theta characteristics] If $\cc(-2d)$ is an even theta characteristic, then $\delta=1$ and we have
 \begin{equation}\label{eq even theta} \mathcal A=\mathcal{O}^{2d} \oplus  \bigoplus_{j=1}^r ({\mathcal O}(k_j)\oplus {\mathcal O}(1-k_j) ) \end{equation} for some non-negative integers $ k_j\geq 0$.
\item[Odd theta characteristics] If $\cc(-2d)$ is an odd theta characteristic then $\delta=1$ and we have
 \begin{equation}\label{eq odd theta} \mathcal A=\mathcal{O}^{2d-3}\oplus \mathcal{O}(-1) \oplus  \bigoplus_{j=1}^r ({\mathcal O}(k_j)\oplus {\mathcal O}(1-k_j) ) \end{equation} for some non-negative integers $k_j\geq 0$.
\end{description}

\subsection{Quadric bundles via symmetric free resolutions}
Recall that every bundle $\mathcal T^{\vee}$ on $\mathbb{P}^2$ admits a free resolution of length two
$0\to \mathcal B\to \mathcal A \to\mathcal{ T}^{\vee}\to 0.$ In this section, we consider quadric bundles given by sections of $\operatorname{Sym}^2 \mathcal T^{\vee}(\delta)$, for $\delta$ satisfying the additional assumption $H^2(\bigwedge^2 \mathcal B(\delta))=0$. 
We prove that such quadric bundles are the quadric reductions of quadric bundles given by  
$(\operatorname{Sym}^2 (\mathcal A\oplus \mathcal B^{\vee}(-\delta)))(\delta)$. Furthermore, the latter quadric bundles need to be included in the classification of Section \ref{free resolutions}. This provides us a convenient way to construct and  investigate many quadric bundles. 
 
\begin{prop}\label{quadric reduction} Let  $\mathcal{T}$ be a vector bundle equipped with a symmetric map $$\varphi\colon \mathcal{T}(-\delta)\to \mathcal{T}^{\vee}$$ inducing a quadric bundle ${Q}$. Let
$0\to \mathcal B\to \mathcal A \to\mathcal{ T}^{\vee}\to 0,$
be a length two resolution of $\mathcal{ T}^{\vee}$ with $\mathcal A$ and $\mathcal B$ split. Assume that $H^2(\bigwedge^2 \mathcal B(\delta))=0$.  Then there exists a symmetric map $\psi\colon \mathcal A^{\vee}(-\delta)\oplus \mathcal B\to \mathcal A\oplus \mathcal B^{\vee}(-\delta)$ which induces a quadric bundle $\mathcal{P}$ such that ${Q}$ is obtained from 
$\mathcal{P}$ by quadric reduction.
\end{prop}

\begin{proof} {\color{black} For the proof, we find an appropriate  decomposition of the space of sections $H^0(\Sym^2( \mathcal A\oplus  \mathcal B^{\vee}(-\delta))(\delta))$ that allows us to define $\psi$. 
First, we have $$H^0(\Sym^2( \mathcal A\oplus  \mathcal B^{\vee}(-\delta))(\delta))=H^0(\Sym^2  \mathcal A(\delta))\oplus H^0( \mathcal A\otimes \mathcal  B^{\vee})\oplus H^0(\Sym^2  \mathcal B^{\vee}(-\delta)).$$

Then, from the vanishing  $H^2(\bigwedge^2 \mathcal B(\delta))=0$}, the short exact sequence $$0\to \textstyle \bigwedge^2 \mathcal B\to \mathcal A\otimes \mathcal B\to \Sym^2  \mathcal A\to \Sym^2 \mathcal T^{\vee}\to 0$$ gives a decomposition  $$H^0(\Sym^2  \mathcal A(\delta))=H^0(\Sym^2 \mathcal T^{\vee}(\delta))\oplus H^0(( \mathcal A\otimes  \mathcal B(\delta))/\textstyle \bigwedge^2  \mathcal B(\delta)).$$  Additionally, by the cohomology exact sequence associated with the resolution of the bundle $\mathcal T^{\vee} \otimes  \mathcal B^{\vee}(-\delta)$ and the vanishing $H^1( \mathcal B\otimes  \mathcal B^{\vee})=0$ (as  $ \mathcal B\otimes  \mathcal B^{\vee}$ is split), we have 
\begin{equation}H^0( \mathcal A\otimes  \mathcal B^{\vee})=
 H^0( \mathcal B\otimes  \mathcal B^{\vee})\oplus H^0(\mathcal T^{\vee}\otimes  \mathcal B^{\vee}).\end{equation}

Finally, we get 
 \begin{eqnarray}
\label{decomposition of sym A+B} 
H^0(\Sym^2( \mathcal A\oplus  \mathcal B^{\vee}(-\delta))(\delta))&=&
H^0(\Sym^2 \mathcal T^{\vee}(\delta))\oplus H^0( \mathcal B\otimes  \mathcal B^{\vee}) \oplus H^0(\mathcal T^{\vee}\otimes  \mathcal B^{\vee}) \oplus \nonumber \\ &\oplus & H^0(( \mathcal A\otimes  \mathcal B(\delta))/\textstyle \bigwedge^2  \mathcal B(\delta))\oplus H^0(\Sym^2  \mathcal B^{\vee}(-\delta)).
\end{eqnarray}
{\color{black} If we now denote by  $\bar {\varphi}$ the element of $H^0(\Sym^2 \mathcal T^{\vee}(\delta))$ corresponding to the symmetric map $\varphi$, we can define a symmetric map 
 $\psi$ corresponding to an element of $H^0(\Sym^2( \mathcal A\oplus  \mathcal B^{\vee}(-\delta))(\delta))$ obtained by adding to $\bar{\varphi}$ the identity on $H^0( \mathcal B\otimes  \mathcal B^{\vee})$ and taking the remaining components to be zero.} In this way, we get a quadric bundle $\mathcal P$
 for which  $ \mathcal B^{\vee}(-\delta)$ is regular isotropic and the associated quadric reduction is exactly $Q$.  
\end{proof}

%The purpose of Lemma \ref{quadric reduction of general map} is to construct quadric reductions from quadric bundles of high rank in order to obtain quadric surface bundles associated with the same Brauer class.  

\begin{lem}\label{quadric reduction of general map} Let $ \mathcal A$, $ \mathcal B$ be split bundles. Let $\delta\in \mathbb{Z}$ be such that $H^2(\bigwedge^2  \mathcal B(\delta))=0$. Assume that there exists an embedding $ \mathcal B\to  \mathcal A$ inducing an epimorphism $$\Hom( \mathcal A, \mathcal B^{\vee}(-\delta))\twoheadrightarrow \Hom( \mathcal B, \mathcal B^{\vee}(-\delta)).$$ Consider a general map $\psi\colon  \mathcal A^{\vee}(-\delta)\oplus  \mathcal B\to     \mathcal A\oplus  \mathcal B^{\vee}(-\delta)$ and assume it to be of generic maximal rank. Then there exists an embedding $\mathcal B\hookrightarrow  \mathcal A^{\vee}(-\delta)\oplus  \mathcal B$ such that $ \mathcal B$ is regular isotropic with respect to $\psi$. Moreover, if we apply the quadric reduction induced by $ \mathcal B$ to the quadric bundle $Q_{\psi}$ associated with $\psi$, we obtain  a quadric bundle $Q_{\phi}$ associated with a symmetric map  $\phi\colon \mathcal{ T}(-\delta)\to \mathcal{T}^{\vee}$, where $\mathcal{ T}$ is a vector bundle fitting in the short exact sequence: $0\to  \mathcal B\to  \mathcal A \to\mathcal{ T}^{\vee}\to 0$.

\end{lem}
\begin{proof}
By generality assumption, in the decomposition (\ref{decomposition of sym A+B})  the component of $\psi$ in $H^0( \mathcal B\otimes  \mathcal B^{\vee})$  can be assumed to be of maximal rank; a change of coordinates allows us to assume it is the identity.  Under our assumptions, the epimorphism $\phi$ allows us to replace  $\psi$ with a conjugate $\theta\psi \theta ^{\vee}$, where $\theta$ is an automorphism  of  $ \mathcal A\oplus  \mathcal  B^{\vee}(-\delta)$  such that the restriction of $\psi$ to $H^0(\Sym^2  \mathcal B^{\vee}(-\delta))$ is 0. Accordingly, $ \mathcal B$ is regular isotropic. The associated quadric reduction is a quadric bundle in the projectivization of the dual of the cokernel of a map $ \mathcal B\to  \mathcal A$. This proves the claim.
\end{proof}

\begin{cor}\label{resolution easy bundles} Let  $\mathcal{T}$ be a vector bundle and  let
$0\to \mathcal B\to  \mathcal A \to\mathcal{ T}^{\vee}\to 0$
be a length two free resolution of $\mathcal{ T}^{\vee}$. Assume that $H^2(\bigwedge^2  \mathcal B(\delta))=0$ for some $\delta\in \mathbb Z$. Let $(C, \cc) $ be a general plane curve of degree $2d$ equipped with a quasi half-period. Then there exists a symmetric map $\varphi\colon \mathcal{T}(-\delta)\to \mathcal{T}^{\vee}$ giving a resolution $0\to \mathcal{T}(-\delta)\to \mathcal{T}^{\vee}\to \cc\to 0$ if and only if $(\mathcal A\oplus  \mathcal  B^{\vee}(-\delta), \delta)$ appears in the classification of Subsection  \ref{free resolutions}.
\end{cor}

\begin{proof} Apply Proposition \ref{split resolution} and Proposition \ref{quadric reduction} and Lemma \ref{quadric reduction of general map}. 
\end{proof}

\begin{rem} The restriction to the fibers of the map in Lemma \ref{quadric reduction of general map} shows that the cokernel of the symmetric map from $H^0(\Sym^2(\mathcal A\oplus \mathcal B^{\vee}))$ is the same as the cokernel of $q$. Therefore, in Corollary \ref{resolution easy bundles} we obtain a free symmetric resolution of the same sheaf from a locally free symmetric resolution of $\cc$.
\end{rem}

%\begin{defi} {\em We will call $\delta$-easy the bundles $\mathcal{T}$ on $\mathbb{P}^2$ which admit a free resolution $0\to B\to A \to \mathcal{ T}^{\vee} \to 0$ with  $H^2(\bigwedge^2 B(\delta))=0$.}
%\end{defi}

\begin{rem}
Note that the condition $H^2(\bigwedge^2 \mathcal B(\delta))=0$  for $\mathcal B=\bigoplus_{i=1}^n \mathcal{O}(b_i)$ amounts to saying that for all pairs $ 1\leq i < j\leq n$ we have $b_i+b_j+\delta > -3$.
Furthermore, if $\mathcal{T}^{\vee}$ admits a free resolution $0\to\mathcal B\to\mathcal A \to\mathcal{ T}^{\vee}\to 0$  with $\mathcal B$ a vector bundle of rank at most two, then we have $H^2(\bigwedge^2 \mathcal B(\delta))=0$. 
\end{rem}

%
%\begin{cor} Let $(C, \cc) $ be a general plane curve of degree 6 equipped with a half-period. Any two symmetric locally free resolutions of $\cc$ involving easy bundles are related by quadric reductions.
%\end{cor}

%\begin{rem} Note that in Proposition \ref{quadric reduction} the quadric reduction is defined with respect to a split bundle. This implies that we can decompose it into a sequence of quadric reductions with respect to sections, i.e. hyperbolic reductions. It is natural to ask for a classification of all symmetric locally free resolutions of sheaves supported on smooth plane sextics, in particular which of them are related by quadric reductions and which are related by hyperbolic reductions.  We hope to address the problem of classification of symmetric resolutions in a future work.
%\end{rem}

\section{Application to rationality problems}\label{rat} 
Consider the map  from the set of quadric surface bundles over $\PP^2$ to the union of moduli spaces of plane curves equipped with quasi half-periods given by taking the associated cokernel. Let us call it the {\em cokernel map}. After fixing the degree of the curve and the type of quasi half-period, the latter moduli spaces are well studied. In particular, they are known to be irreducible (cf. \cite[$\S$ 4]{B}, \cite[Propositions 2 and 3]{B2}). %The map is given by taking the associated cokernel. 
Through the cokernel map a family of quadric surface bundles is mapped to a family of  pairs $(C,\cc)$. Our aim in this section is to provide families of quadric surface bundles which are mapped via the cokernel map to families that dominate moduli spaces of plane curves equipped with quasi half-periods. We will call them \emph{dominant families}. We pursue the following goals:
\begin{itemize}
\item to provide alternative models for dominant families of quadric surface bundles, the rationality of which is not known, such as Verra fourfolds or cubic fourfolds containing a plane;
\item to provide dominant families of quadric surface bundles such that very general quadric surface bundles in these families are not stably rational and to deduce irrationality of any quadric surface bundle whose associated cokernel is very general in the corresponding  moduli spaces; 
\item to provide dominant families of quadric bundles with rational general element, so that all quadric bundles with associated cokernel in the corresponding moduli space will be rational.
\end{itemize}

The method of quadric reduction will be our main tool to construct quadric surface bundles of each type.
\begin{rem}
For any quadric surface bundle one can always construct a birational model as follows. Consider a cokernel sheaf $\cc$ associated with $Q$  and take  one of its free symmetric resolutions $\mathfrak R$. This exists regardless of the genericity assumption. Consider the quadric bundle $Q_{\mathfrak R}$ associated with the free resolution $\mathfrak R$ and perform quadric reductions of $Q_{\mathfrak R}$ with respect to isotropic sections (cf. Remark \ref{reduction to rank 4}) until we get a  (weak) quadric surface bundle. As a consequence of Theorem \ref{main} and Corollary \ref{birational quadric bundle}, the latter quadric surface bundle will be birational to $Q$. Summing up, any quadric bundle has a birational model which is obtained as a quadric reduction of a free resolution of some quasi half-period.
 Although this can be made explicit in a chosen example, in what follows we are interested in constructing families of quadric surface bundles and hence we look at families of quadric reductions, which is a much more complicated problem requiring additional assumptions.
 \end{rem}

In order to perform a quadric reduction, we need to find a regular isotropic sub-bundle of sufficiently high rank. For the sake of simplicity, we investigate split isotropic sub-bundles of quadric bundles in projectivizations of split bundles associated with free symmetric but not necessarily minimal resolutions introduced in Section \ref{ex} (cf. Section \ref{free resolutions}). Let $Q$ be a quadric bundle in the projectivization of a split bundle $\PP(\mathcal V)$. It is associated with a symmetric matrix $A$ of polynomials. The embedding of a projectivization of a split bundle $\PP(\mathcal U)$ in   $\PP(\mathcal V)$ is determined by a polynomial matrix $N$ representing a map $N\colon\mathcal V\to \mathcal U$. The matrix $N$ determines a regular isotropic sub-bundle of $Q$ if $N^TAN=0$ and both $N$ and $AN$ are everywhere of maximal rank.
We need only to consider the cases from the classification in Section \ref{free resolutions}.

The existence of suitable quadric reductions will follow from a rank computation in each of the cases above.
\begin{lem} \label{lem rank} Let us fix non-negative integers, $r$, $l$, $k$. Let $M\in \Hom(\mathcal O^l, \mathcal O(2)^r)$ and $Z\in H^0(\Sym^2(\mathcal O^l)(2))$ be general. Then the map 
$$
\begin{array}{ccc}
 \Hom(\mathcal O^r, \mathcal O^l)\times \Hom(\mathcal O^r, \mathcal O(1)^{k}) & \overset{\eta}{\longrightarrow} & H^0(\Sym^2(\mathcal O^r)(2)) \\
 & & \\
 (K,N) & \mapsto & K^T M^T+M K+ K^TZK+ N^TN
\end{array}
$$
is dominant if $k=4-l$ and one of the following holds:  $r\leq 3$ or ($r\leq 5$ and $l\leq 3$) or ($r\leq 6$ and $l\leq 2$) or ($r\leq 8$ and $l\leq 1$) or ($r=9$ and $l=0$). 
\end{lem}
\begin{proof}
This is proven by means of Macaulay 2. A script that can be used to perform the computations can be found in Appendix A {\color{black}and \cite{M2GH}}. The map $\eta$ is a map between affine spaces that is given by linear and quadratic polynomials. Using Macaulay 2, we compute the rank of its differential over chosen points and always find a regular point of $\eta$ (if we stay in the assumed range); hence $\eta$ is dominant. Since our goal is to prove the theorem for any field of characteristic different from 2, we work over the integers and compute the greatest common divisor of several maximal minors of the differentials of $\eta$ in several points defined over $\mathbb Z$. In each case, we manage to obtain a greatest common divisor which is a power of 2. This implies that we can find a point at which $\eta$ is regular in every characteristic different from 2.
\end{proof}

We can now prove the existence of quadric reductions by a dimension count. 

\begin{cor} \label{cor half period} Let $A$ be a general symmetric map
$A\colon \mathcal{O}(-1)^d\oplus \mathcal O^k \to \mathcal O(1)^d\oplus \mathcal O^k$. Then for $(d,k)$ such that $d+k$ is even and the triple $(r,l,k)=(\frac{d+k}{2}-2, 2+\frac{d-k}{2},k)$ satisfies the condition in Lemma \ref{lem rank}, there exists a regular isotropic sub-bundle $$\mathcal{O}(-1)^{\left(\frac{d+k}{2}-2\right)}\simeq \mathcal U\subset \mathcal{O}(-1)^d\oplus \mathcal O^k$$ with respect to $A$. Moreover, the quadric reduction with respect to ${\mathcal U}$ is a quadric bundle $Q^{A}_{{\mathcal U}}$ given by a general map $\mathcal T\to \mathcal T^{\vee}$, where $\mathcal T$ is described by the following exact sequence:
$$0\to \mathcal{O}(-1)^{\left(\frac{d+k}{2}-2\right)}  \to \mathcal O(1)^{\left(\frac{d-k}{2}+2\right)} \oplus \mathcal O^k\to \mathcal T^{\vee}\to 0.$$
\end{cor}

\begin{proof} Denote by $r$ the integer $\frac{d+k}{2}-2$. Note that by change of coordinates we may assume $A=A'\oplus \operatorname{id}$, where $A'$ is a map $\mathcal{O}(-1)^d\to \mathcal O(1)^d$ given by $d\times d$ matrix of quadric polynomials and $\operatorname{id}:\mathcal O^k\to \mathcal O^k$ is the identity represented by the identity matrix $\operatorname{Id}_k$. We look for matrices $U$ of size $r\times (d+k)$ consisting of a block $K'$ of size $r\times d$ of constants and a block $N$ of size $r\times k$ of linear forms. In other words, we have a block decomposition:
$$A= \left(\begin{array}{c|c}A'& 0\\
\hline
0& \operatorname{Id}_k
\end{array}\right), \quad
U= \left(\begin{array}{c}K'\\
\hline
N
\end{array}\right).
$$

The isotropy condition $U^TAU=0$ translates to $K'^TA'K'+N^TN=0$. Let us now  split further $A'$ into a symmetric $r\times r$ block $A''$, another symmetric $(d-r)\times (d-r)$ block $Z$ and an $r\times (d-r)$ block $M$ and assume that $K'$ consists of an identity block and an $r\times (d-r)$
 block $K$: 
$$A'= \left(\begin{array}{c|c}A''& M\\
\hline
M^T& Z
\end{array}\right), \quad
K'= \left(\begin{array}{c}\operatorname{Id}_r\\
\hline
K
\end{array}\right).
$$
 
Then the isotropy condition translates to:
 $$A''+MK+(MK)^T+ K^TZK+N^TN=0.$$ Since $A'$ is general, the map 
$$(K, N)\mapsto MK+(MK)^T+ K^TZK+N^TN$$
is dominant by Lemma \ref{lem rank} and the corresponding equation has a solution. Since $N$ is everywhere of maximal rank, it defines a subbundle  ${\mathcal U}$ as in the statement, which can be assumed to be smooth by genericity. The description of the quadric reduction and its geometric interpretation follow easily.
\end{proof}

\begin{rem}
From a geometric point of view, the quadric surface bundles  $Q^{A}_{\mathcal U}$ in Corollary \ref{cor half period} are described as complete intersections of $\frac{d+k}{2}-2$ divisors of class $\xi +h$ and one divisor of class $2\xi$ in the projective bundle $\mathbb{P}\left(\mathcal O_{\mathbb P^2}(1)^{\left(\frac{d-k}{2}+2\right)}\oplus \mathcal O^k_{\mathbb P^2}\right)$, where $\xi$ is the relative hyperplane class and $h$ is the class of the pullback of a line in $\mathbb{P}^2$. In particular, for $k=d+4$ the quadric bundle $Q^{A}_{\mathcal U}$ is a complete intersection of $d$ divisors of bidegree $(1,1)$ and one divisor of bidegree $(0,2)$ in $\mathbb{P}^2\times \mathbb{P}^{d+3}$.
\end{rem}

\begin{cor}\label{cor even theta} Let $A$ be a general symmetric map $A\colon \mathcal{O}(-1)^{2d} \to \mathcal O^{2d}$ for $d\leq 7$. Then there exists a sub-bundle $\mathcal{O}(-1)^{d-2}\simeq \mathcal  U\subset \mathcal{O}(-1)^{2d}$ that is regular isotropic with respect to $A$. The corresponding quadric reduction is a quadric surface bundle $Q^{A}_{\mathcal U}$ given by a general map $\mathcal T(-1)\to \mathcal T^{\vee}$ such that
$$0\to \mathcal{O}(-1)^{d-2} \to \mathcal O^{d+2} \to \mathcal T^{\vee}\to 0.$$ Geometrically, $Q^{A}_{\mathcal U}$ is a complete intersection of $(d-2)$ divisors of bidegree $(1,1)$ and one divisor of bidegree $(1,2)$ in $\mathbb{P}^2\times \mathbb P^{d+1}$.

\end{cor}
\begin{proof} Consider the block decomposition of $$A= \left(\begin{array}{c|c}A'& M\\
\hline
M^T& Z 
\end{array}\right),
$$
where $A'$ is a $(d-2)\times (d-2)$ diagonal block. We will look for matrices $U$ of the form  $$U=\left(\begin{array}{c}\operatorname{Id}_{d-2}\\
\hline
N
\end{array}\right),
$$ such that $U^TAU=0$. This amounts to checking that for a general $A'$ the matrix  $-A'$ is in the image of the map $N\mapsto MN+N^TM^T+N^TZN$. This map is dominant for $d\leq 7$, which can be checked similarly to Lemma \ref{lem rank} with the use of a script analogous to that of Appendix A and \cite{M2GH}.
\end{proof}

\begin{cor}\label{cor odd theta} Let $A$ be a general symmetric map $$A\colon  \mathcal{O}(-1)^{2d-3}\oplus \mathcal O(-2) \oplus (\mathcal O\oplus \mathcal O(-1) )^k \to \mathcal O^{2d-3}\oplus \mathcal O(1) \oplus (\mathcal O\oplus \mathcal O(-1) )^k
$$ for $d\leq 6$ and $k\in \{0,1\}$. Then there exists a sub-bundle $$ \mathcal{O}(-1)^{(d+k-3)} \oplus \mathcal O(-2)^k \simeq U\subset  \mathcal{O}(-1)^{(2d-3)}\oplus \mathcal O(-2)$$ regular isotropic with respect to $A$. The quadric reduction with respect to $U$ is a quadric surface bundle $Q^{A}_{\mathcal U}$ given by a general map $\mathcal T(-1)\to \mathcal T^{\vee}$ such that 
$$0\to \mathcal{O}(-1)^{d-3} \oplus \mathcal O(-2)^k \to  \mathcal O^{d+k} \oplus \mathcal O(1)^{1-k} \oplus \mathcal O(-1)^k \to \mathcal T^{\vee}\to 0.$$
\end{cor}

\begin{proof} The proof is analogous to that of Corollary \ref{cor half period} and Corollary \ref{cor even theta}.
\end{proof}

\begin{rem}
Geometrically, the quadric bundles $Q^{A}_{\mathcal U}$ from Corollary \ref{cor odd theta}  are complete intersections of $(d-3)$ divisors of class $\xi+h$, $k$ divisors of class $\xi+2h$ and one divisor of class $2\xi+h$ in $\mathbb{P}\left( \mathcal O^{d+k} \oplus \mathcal O(1)^{1-k} \oplus \mathcal O(-1)^k\right)$. In particular, for $k=0$  the bundles $Q^{A}_{\mathcal U}$ are birational to residual components of intersections of $(d-3)$ quadrics and a cubic containing a $\mathbb{P}^{d-1}$ in $\mathbb{P}^{d+2}$.
\end{rem}

Corollary \ref{cor half period}, Corollary \ref{cor even theta}, Corollary \ref{cor odd theta} allow us to describe birational models for quadric surface bundles over $\mathbb{P}^2$ with discriminant a general smooth curve $C$ of degree $2d \leq 12$. In particular, we can now easily deduce Proposition \ref{qw}.
\begin{comment}
 \begin{prop}\label{qw} Let $C$ be a general smooth plane curve of degree $2d\leq 12$. Let $Q$ be a quadric surface bundle over $\mathbb{P}^2$ with discriminant $C$. Then $Q$ {\color{black}has isomorphic generic fiber over $\mathbb{P}^2$ with} one of the following:
\begin{enumerate}
\item a trivial quadric surface bundle,
\item a complete intersection of $d$ divisors of bidegree $(1,1)$ and one divisor of bidegree $(0,2)$  in $\mathbb{P}^2\times \mathbb{P}^{d+3}$,
\item a complete intersection  of $d-2$ divisors of bidegree $(1,1)$ and one divisor of bidegree $(1,2)$ in $\mathbb{P}^2\times \mathbb P^{d+1}$,
\item the residual component  of an intersection of $d-3$ quadrics and a cubic containing a $\mathbb{P}^{d-1}$ in $\mathbb{P}^{d+2}$.
\end{enumerate}
\noindent Conversely, a general curve $C$ appears as a discriminant of some quadric bundle in any of the $3$ families above. 

\end{prop}
\end{comment}

\begin{proof}[Proof of Proposition \ref{qw}]% If $C$ is a general curve, up to twist by line bundles on $\mathbb{P}^2$, there are $4$ types of quasi half-periods on $C$. These are: the trivial bundle,  half-periods, even or odd theta characteristics. As proved in \cite{C}, each of these admits a free symmetric resolution associated to a symmetric map $A$ as in one of the Corollaries \ref{cor half period}, \ref{cor even theta}, \ref{cor odd theta} with $k=0$. For such $A$ we can construct  a quadric surface bundle $Q^A_{\mathcal U}$ related to it and belonging to one of the families listed in the assertion.

If $Q$ is any quadric bundle with discriminant $C$, it must be associated with a symmetric map between a vector bundle and a twist of its dual. The cokernel of such a map is a quasi half-period on $C$, which after a twist by some line bundle $\mathcal O_C(\gamma)$ is a line bundle $\mathcal L$ of one of the four following types: the trivial bundle, a half-period, an even theta characteristic, an odd theta characteristic.
 As proved in \cite{C}, each of the non-trivial cases admits a free symmetric resolution associated to a symmetric map $A$ as in the assumptions of Corollary \ref{cor half period}, Corollary \ref{cor even theta}, Corollary \ref{cor odd theta} for $k=0$. By extending $A$ - if possible - by the identity map to some map $A_k$, we have a suitable resolution of $\cc$  for any $k$. In the assumed ranges we can use Corollary \ref{cor half period}, Corollary \ref{cor even theta}, Corollary \ref{cor odd theta} to construct  a quadric surface bundle $Q^{A_k}_{\mathcal U}$ related to $A_k$ and belonging to one of the families listed in the statement.
By  Theorem \ref{main}, we deduce that $Q$ and $Q^{A_k}_{\mathcal U}$ are birational.
%o compare $Q$ with the quadric bundle $Q_A$ associated by Corollaries \ref{cor half period}, \ref{cor even theta}, \ref{cor odd theta} to a free resolution $R_A$ of  $\mathcal L$.  As a consequence, the quadric surface bundle $Q$ is birational over $\mathbb{P}^2$ to one of the quadric surface bundles from the families listed in the assertion.
\end{proof}

Using the same methods we can also construct more birational models of the same quadric surface bundles. 
\subsection{Birational models of quadric bundles with discriminant of degree $6$}\label{deg6}
In this section we aim to construct different birational models of the quadric surface bundles in Proposition \ref{qw}.

Fix $A$ giving a minimal symmetric free resolution of a half period or a theta characteristic. Consider $A_k=A\oplus \operatorname{id}$ for some $k$ such that $(A_k,k)$ satisfy the assumption of {\color{black} Corollary \ref{cor half period} or Corollary \ref{cor odd theta}.} Then by Theorem \ref{main} we obtain different birational models of the same quadric surface bundle. More precisely, regardless of the choice of $\mathcal U_k$ and $\mathcal U$ in Corollary \ref{cor half period} or Corollary \ref{cor odd theta}, the quadric bundle ${Q}^{A_{k}}_{\mathcal U_k}$ will be birational  to ${Q}^{A}_{\mathcal U}$. As a consequence, we can find the following birational models of quadric bundles with discriminant curves of degree $6$. Similar constructions can be performed for bundles with discriminant of higher degree.

\begin{example}\label{po} Corollary \ref{cor half period} gives rank $4$ symmetric resolutions for a general curve $C$ of degree $6$ equipped with a half-period $\cc$. More precisely, there exists a symmetric resolution
  $0\to \mathcal{T}\to \mathcal{T}^{\vee}\to \cc(3)\to 0,$ whenever $\mathcal T^{\vee}$ has one of the following resolutions:
  \begin{enumerate}
  \item $0\to \mathcal O(-1)^3\to \mathcal O^7 \to\mathcal{ T}^{\vee}\to 0$,
  \item $0\to \mathcal O(-1)^2\to \mathcal O^5 \oplus \mathcal O (1)\to\mathcal{ T}^{\vee}\to 0$,
  \item $0\to \mathcal O(-1)\to \mathcal O^3 \oplus \mathcal O (1)^2\to\mathcal{ T}^{\vee}\to 0$,
  \item $\mathcal{ T}^{\vee}=\mathcal O \oplus \mathcal O(1)^3$.
  \end{enumerate}
In each of the cases above, the quadric bundles have natural geometric descriptions providing alternative birational models of general Verra fourfolds. We have the following corresponding quadric fibrations.
Let $\xi$ be the class of the Grothendieck bundle $\mathcal{O}_{\PP(\mathcal A)}(1)$ and  $h$ the pullback of the class line from $\mathbb{P}^2$ through the fibration of $\mathbb{P}(\mathcal A)$, where as in Section \ref{ex}, $\mathcal A$ is defined via the resolution $ 0\to\mathcal  B\to \mathcal A \to\mathcal{ T}^{\vee}\to 0.$

 \begin{enumerate}
  \item $\mathcal A=\mathcal O^7$ and the quadric fibration $Q_1$  is a complete intersection of three divisors of class $\xi+h$ and one divisor of class $2\xi$. Thus, we get a complete intersection n $\PP^2\times \PP^6$ of three divisors of bidegree $(1,1)$ and one divisor of bidegree $(0,2)$.
  \item $\mathcal A=\mathcal O^5\oplus \mathcal O(1)$  and $Q_2$ is a complete intersection of two divisors of class $\xi+h$ and one divisor of class $2\xi$. This case will be discussed more precisely in Section \ref{1.4} (cf. Lemma \ref{5.3}).
  \item $\mathcal A=\mathcal O^3\oplus \mathcal O(1)^2$  and $Q_3$ is a complete intersection of one divisor of class $\xi+h$ and one divisor of class $2\xi$.
  \item $\mathcal A=\mathcal T^{\vee}=\mathcal O \oplus \mathcal O(1)^3$ and $Q_4$ is a divisor of class $2\xi$. Since the system $|\xi|$ maps $\mathbb{P}(\mathcal O \oplus \mathcal O(1)^3)$ to a cone over $\mathbb{P}^2\times \mathbb{P}^2$ the quadric bundle $Q_4$ is a general Verra fourfold (cf. Lemma \ref{Verra}). \end{enumerate}
\end{example}

\begin{example}Let us also give an example of a symmetric resolution involving a bundle $\mathcal T$ that does not satisfy the assumptions of Corollary \ref{resolution easy bundles}. Let $\mathcal T$ be such that its dual $\mathcal T^{\vee}$  admits the following resolution:
$$0\to \mathcal O(-2)\oplus\mathcal O(-1)\to \mathcal O^6 \to\mathcal{ T}^{\vee}\to 0.$$ Then $\mathrm{Sym}^2 T$ has $21$ sections and a general such section gives an exact sequence
 $0\to \mathcal{T}\to \mathcal{T}^{\vee}\to \cc\to 0$ for some $\cc$ representing a half-period of a smooth discriminant curve $C$. The corresponding quadric fibration gives a fourfold which is birational to some special Verra fourfold. The family of such fourfolds has dimension $21-9=12.$ The elements of this family are described as complete intersections of three divisors of respective bi-degrees  $(1,2)$, $(1,1)$, $(2,0)$ in $\mathbb{P}^5\times \mathbb{P}^2$.
 \end{example}

\begin{example}\label{2.14} For $\cc$ a theta characteristic on a general curve $C$ of degree $6$, we have a symmetric resolution  $$0\to \mathcal{T}(-1)\to \mathcal{T}^{\vee}\to \cc(1)\to 0,$$ whenever  $\mathcal T^{\vee}$ has one of the following resolutions:
\begin{itemize}
  \item [(a)] $0\to \mathcal O(-1)\to \mathcal O^5 \to\mathcal{ T}^{\vee}\to 0$ and $h^0(\cc)=0$;
   \item [(b)] $0\to \mathcal O(-2)\to \mathcal O^4 \oplus \mathcal O(-1)\to\mathcal{ T}^{\vee}\to 0$ and $h^0(\cc)=1$;
   \item [(c)] $\mathcal{ T}^{\vee}=\mathcal O^3 \oplus \mathcal O(1)$ and $h^0(\cc)=1$.
\end{itemize}
A geometric description of the associated quadric fibration is given as follows.
 \begin{itemize}
  \item [(a)] $\mathcal A= \mathcal O^5$  and the quadric fibration $Q_1$  is a complete intersection of two divisors of class $\xi+h$ and one divisor of class $2\xi+h$. Thus, we get a complete intersection in $\PP^2\times \PP^4$ of one divisor of bidegree $(1,1)$ and one  divisor of bidegree $(1,2)$.
  \item [(b)] $\mathcal A=\mathcal O^4 \oplus \mathcal O(-1)$  and $Q_2$ is a complete intersection of two divisors of class $\xi+2h$ and one divisor of class $2\xi+h$. The image is naturally embedded in the cone $C(\PP^2\times \PP^3)\subset \PP^{11}$.
  \item [(c)] $\mathcal A=\mathcal O^3\oplus \mathcal O(1)$ and $Q_3$  is a divisor of class $2\xi+h$. Here $Q_3$ is mapped isomorphically to a cubic containing  a plane in $\mathbb{P}^5$.
\end{itemize}
\end{example}
Summing up, we obtain the following birational models of quadric surface bundles over $ \mathbb{P}^2$.
\begin{cor}\label{pp} Let $Q$ be a quadric surface bundle over $\mathbb{P}^2$ whose associated cokernel is a twist of an even theta characteristic on a smooth sextic curve. Then $Q$ is rational and {\color{black} is isomorphic at the generic point of  $\mathbb{P}^2$ to} a complete intersection of two divisors of bi-degrees $(1,1)$,$(1,2)$ in $\PP^2\times \mathbb P^4$.
\end{cor}
\begin{proof} By Proposition \ref{qw}, a quadric bundle whose associated cokernel is an even theta characteristic on a general curve of degree $6$  is birational to some quadric surface bundle $Q_1$ from Example \ref{2.14} (a) (cf. Proposition \ref{qw}(2) for $d=3$). The latter is rational, as the natural projection from the quadric bundle to $\PP^4$ is a birational morphism.
\end{proof}
\begin{rem} Corollary \ref{pp} provides examples of quadric surface bundles which are not birational over $\mathbb{P}^2$ to the trivial bundle (since the associated Brauer class is non-trivial) but are still rational. 
\end{rem}

%Theorem \ref{main} can also be used to provide different birational models of Verra fourfolds and cubics containing a plane.
Let $\xi_i$, for $i=1\dots 6$, in each case below denote the class of the Grothendieck bundle $\mathcal{O}_{\PP(\mathcal A_i)}(1)$ and  $h_i$ be the pullback of the class of a line from $\mathbb{P}^2$ through the fibration of $\mathbb{P}(\mathcal A_i)$.
\begin{corr}\label{corr} Consider the following families $\mathfrak{F}_i$ of varieties.
\begin{enumerate}
%\item  $\mathfrak{F}_1 =\{\text{nodal Gushel--Mukai  fourfolds}\}$;

\item $\mathfrak{F}_1 =\{\text{complete intersections of three $(1,1)$  divisors in $\PP^2\times \mathbf Q_5$}\}$, \\where $\mathbf{Q}_5$ stands for a smooth five-dimensional quadric;
\item $\mathfrak{F}_2 =\{\text{complete intersections of three divisors in   $\mathbb{P}(\mathcal A_2)=\mathbb P(\mathcal O_{\mathbb P^2}^5\oplus \mathcal O_{\mathbb P^2}(1))$,} \\
\text{two of class $\xi_2+h_2$ and one of class $2\xi_2$}\}$,  %where $\xi$ is the class of $\mathcal O_{\PP(\mathcal{O}_{\mathbb P^2}^5\oplus \mathcal{O}_{\mathbb P^2}(1))}(1)$ and $h$ is the pullback of the class of a line in $\mathbb{P}^2$ via the projection to the base of the projective bundle. 
  \item $\mathfrak{F}_3 =\{\text{complete intersections of two  divisors in $\mathbb{P}(\mathcal A_3)=\PP(\mathcal{O}_{\mathbb P^2}^3\oplus \mathcal{O}_{\mathbb P^2}(1)^2)$ }\\\text{of respective class } \xi_3+h_3 \text{ and }2\xi_3 \}$, 
  %where $\xi$ is the class of $\mathcal O_{\PP(\mathcal{O}_{\mathbb P^2}^3\oplus \mathcal{O}_{\mathbb P^2}(1)^2)}(1)$ and $h$ is the pullback of the class of a line in $\mathbb{P}^2$ via the projection to the base of the projective bundle. 

  \item $\mathfrak{F}_4 =\{\text{Verra fourfolds}\}$;
\end{enumerate}
 Then for every $i,j\in \{1\dots 4\}$ a general element of $\mathfrak{F}_i$ is birational to some element of $\mathfrak{F}_j$. Moreover, consider the following two families of varieties.
 \begin{enumerate}
 \item $\mathfrak{F}_5$=$\{\text{cubic fourfolds  containing a plane}\}$;
 \item $\mathfrak{F}_6=\{\text{complete intersections of divisors of respective class } 2\xi_6+h_6, \xi_6+2h_6 \text{ on } \\
 \mathbb{P}(\mathcal A_6)=\mathbb{P}(\mathcal O_{\mathbb P^2}^4 \oplus \mathcal O_{\mathbb P^2}(-1))\},$
 %where $\xi$ is the class of $\oo_{\PP(\mathcal O_{\mathbb P^2}^4 \oplus \mathcal O_{\mathbb P^2}(-1))}(1)$ and $h$ is the pull back of the class of a line from $\PP^2$ via the projection to the base of the projective bundle.
  \end{enumerate}
  Then a general element of $\mathfrak{F}_5$ is birational to some element of $\mathfrak{F}_6$ and conversely a general element of $\mathfrak{F}_6$ is birational to some element of $\mathfrak{F}_5$.
\end{corr} 
\subsection{Irrationality of quadric bundles; Proof of Corollary  \ref{cor irrationality}} \label{irrat} 
In this section, we work over the field of complex numbers.

For $d=2$ the quadric surface bundles are all rational, as our models are in fact quadric surface bundles contained in projectivizations of split bundles, which are rational for $d=2$: see, for instance, \cite{Sch1}. 

For $d=3$, among the three non-trivial types of quadric bundles with general discriminant there exists two families for which rationality is an open problem. These are Verra fourfolds and cubic fourfolds containing a plane. The remaining types consist of rational quadric surface bundles described by Corollary \ref{pp}.

%\begin{lem}\label{rationality of new case} A general  quadric bundle $Q_1$  from Example \ref{2.14} (a) is rational.
%\end{lem}
%\begin{proof} Observe that the natural projection to $\mathbb{P}^4$ is a birational morphism. 
%\end{proof}

Our aim is to prove the irrationality of very general quadric bundles (in terms of their cokernels) in the range presented in Corollaries \ref{cor half period}, \ref{cor even theta}, \ref{cor odd theta} with the additional condition $d \geq 4$. For the proof, we are looking for degenerations of such families of quadric surface bundles to quadric surface bundles with non-trivial non-ramified cohomology introduced in \cite{HPT}. We do it in several steps. Starting with a moduli space of cokernels
 we first find a degeneration of the family $\{ Q_{\iota}\}$ of higher rank quadric bundles corresponding to resolutions from Section \ref{free resolutions} such that the degeneration is a quadric bundle $Q_0$ admitting a quadric reduction to a quadric surface bundle $Q_F$ as in \cite{HPT}. Then, we perform a global quadric reduction on the family $\{ Q_{\iota}\}$ to quadric surface bundles and compare its limit in the direction of $Q_0$ to $Q_F$ by means of Corollary \ref{birational quadric bundle}. We deduce stable irrationality of the very general element in the constructed family of reductions, which is dominating the chosen moduli space. Finally, by Proposition \ref{qw} this implies stable irrationality of the very general element of any family dominating the chosen moduli space.

%Let $\widetilde Q$ be a quadric surface bundle with smooth very general discriminant curve $C$ of degree $2d$ associated with a symmetric map 
%$\phi\colon \mathcal T(-\delta)\to \mathcal T^{\vee}$
%for some bundle $\mathcal{T}$. Let $\cc$ be the cokernel sheaf of $\phi$, which is a quasi half-period on $C$. It follows that $\cc$ admits one of the free symmetric resolutions described in Section \ref{free resolutions}.
%We will construct families of suitable quadric reductions of the corresponding higher dimensional quadric bundles using the 

The first step is split into two Lemmas.

\begin{lemma}\label{degeneration from extension} Let $\mathcal A$ be a vector bundle and $\delta\in \mathbb Z$.
Let $\{{Q}_{\iota}\}_{\iota\in I}$ with $Q_\iota \subset \mathbb{P}(\mathcal A^{\vee})$ be the family of all quadric bundles over $\mathbb{P}^2$ corresponding to symmetric maps
$\mathcal A(-\delta)\to \mathcal A^{\vee}$. Assume there exist bundles $\mathcal H$, $\mathcal E$  and an exact sequence:
\begin{equation}\label{exact sequence for degeneration} 0\to\mathcal H^{\vee}\xrightarrow{\eta} \mathcal A^{\vee}\xrightarrow{\theta} \mathcal E^{\vee}\to 0,
\end{equation}
such that:
\begin{enumerate}
\item There exists a symmetric isomorphism $D\colon \mathcal H(-\delta)\to \mathcal H^{\vee}$. 
\item  The restriction map induced by the exact sequence \eqref{exact sequence for degeneration} gives a surjection $H^0(\operatorname{Sym}^2 \mathcal A(\delta))\to H^0(\operatorname{Sym}^2 \mathcal E(\delta))$.
\end{enumerate}

 Let $\phi\colon \mathcal E(-\delta)\to \mathcal E^{\vee}$ be a symmetric map and 
$Q_\phi \subset \mathbb P(\mathcal H^{\vee}\oplus \mathcal E^{\vee} )$ the quadric bundle associated with the symmetric map 
$$(D, \phi)\colon \mathcal H(-\delta)\oplus \mathcal E(-\delta)\to \mathcal H^{\vee} \oplus \mathcal E^{\vee}.$$ Then, for general $\iota\in I$, there exists a flat family of quadric bundles with one fiber $Q_{\iota}$ and another fiber $Q_{\phi}$. 
\end{lemma}
\begin{proof} 

Consider a family $\{\psi_t\}_{t\in \mathbb C}$ of maps:
$\psi_t\colon \mathcal A^{\vee}\oplus \mathcal E^{\vee}\to \mathcal E^{\vee},$ defined as $\psi_t=\theta + t\operatorname{id} $, which gives rise to a family of bundles $\mathcal A^{\vee}_t= \ker \psi_t.$ Observe that $\mathcal A^{\vee}_t\simeq \mathcal A^{\vee}$ for general $t\neq 0$, whereas $\mathcal A^{\vee}_0=\mathcal H^{\vee}\oplus \mathcal E^{\vee}$. We now construct a family $\{\Psi_t\}_{t\in \mathbb C}$ of symmetric maps 
$\Psi_t\colon \mathcal A(-\delta)\oplus \mathcal E(-\delta)\to \mathcal A^{\vee}\oplus \mathcal E^{\vee}$, defined by block matrices 

$$\Psi_t=\left(\begin{array}{c|c}t^3 A+t^2M+\eta D \eta^T & t^2\theta A+ t\theta M\\
\hline
t^2(\theta A)^T+ t (\theta M)^T& \phi+ t \theta A \theta^T 
\end{array}\right),$$
where $\eta, \theta, $ are as in (\ref{exact sequence for degeneration}), $M\colon \mathcal A(-\delta)\to \mathcal A^{\vee}$ is a symmetric map  such that $\theta M \theta^T=\phi$ (which exists by assumption),  $D\colon \mathcal H(-\delta)\to \mathcal H$ is the isomorphism from the assumption, and finally $A$ is a general symmetric map $\mathcal A(-\delta)\to \mathcal A^{\vee}$. Note that $\psi_t\circ \Psi_t=0$; hence after quadric reductions induced by $\{\psi_t\}_{t\in \mathbb C}$ the family of maps $\{\Psi_t\}_{t\in \mathbb C}$ induces a flat family of quadric bundles $Q_t$ on $\mathcal A_t$ such that $Q_t$ is general in $\{ Q_\iota\}_{\iota \in I}$ and $ Q_0=Q_\phi$.
\end{proof}

\begin{lemma}\label{extension of sch} Let  $\mathcal L$ be a half-period, or an even or odd theta characteristic on a general plane curve $C$ of degree $2d$ such that $d$ satisfies the corresponding condition in Corollary \ref{cor irrationality}. Then  we can find a resolution $0\to \mathcal A(-\delta)\to \mathcal A^{\vee}\to \cc\to 0$ as in Section \ref{free resolutions} such that there exists a rank $4$ sub-bundle $\mathcal E\subset \mathcal A$ for which the projection map $
\theta\colon \mathcal{ A}^{\vee}\to \mathcal E^{\vee}$ fits into a short exact sequence $0\to \mathcal H^{\vee} \xrightarrow{\eta} \mathcal A^{\vee}\xrightarrow{\theta} \mathcal E^{\vee}\to 0$ and the following hold:
\begin{enumerate}

\item \label{condition on H} There exists a symmetric isomorphism $\mathcal H(-\delta)\to \mathcal H^{\vee}$.
\item The restriction map induced by the exact sequence induces a surjection \\$H^0(\operatorname{Sym}^2 \mathcal A(\delta))\to H^0(\operatorname{Sym}^2 \mathcal E(\delta))$.
\item There exists a non-degenerate symmetric map $F\colon \mathcal E(-\delta)\to \mathcal E^{\vee}$ such that the associated quadric bundle $Q_F\subset \mathbb{P}(\mathcal E^{\vee})$ has smooth general fiber, non-trivial discriminant in $\mathbb{C}(\mathbb{P}^2)/\mathbb{C}(\mathbb{P}^2)^* $ and $H^2_{nr}(\mathbb{C}(Q_F)/\mathbb C, \mathbb Z_2)\neq 0$.
\end{enumerate}
\end{lemma}

\begin{proof} We will look for sub-bundles $\mathcal H^{\vee}$ of $\mathcal A^{\vee}$  which are of the form $\mathcal H^{\vee}=\mathcal O^r \oplus \mathcal H'\oplus \mathcal (\mathcal H')^{\vee}(\delta) $  for some integer $r$ and some bundle $\mathcal H'$. Note that if $\delta=0$ or $r=0$ such an $\mathcal H^{\vee}$ satisfies condition (\ref{condition on H}). Surjectivity (2) will be checked by the exact sequence:
$$0\to\textstyle \bigwedge^2  \mathcal H (\delta)\to \mathcal H\otimes \mathcal A^{\vee}(\delta) \to \Sym^2 \mathcal A^{\vee}(\delta) \mathcal \to  \Sym^2\mathcal E^{\vee}(\delta)\to 0. $$
Indeed, if $$\mathcal F=\ker(  \Sym^2 \mathcal A^{\vee}(\delta) \mathcal \to  \Sym^2\mathcal E^{\vee}(\delta)),$$ surjectivity of $H^0(\operatorname{Sym}^2 \mathcal A(\delta))\to H^0(\operatorname{Sym}^2 \mathcal E(\delta))$ follows from $H^1(\mathcal F)=0$. 
The latter condition is fulfilled in particular when $H^1(\mathcal H^{\vee}\otimes \mathcal A^{\vee}(\delta) )= H^2(\bigwedge^2 \mathcal H^{\vee}(\delta))=0$. In order to satisfy the last condition, we will consider case by case the following explicit sub-bundles. Below we list - in each case - the bundles $\mathcal A$, $\mathcal H$ and the shape of the  free resolution of the resulting candidate bundles $\mathcal E$. 
\begin{enumerate}
\item In case (\ref{eq half period}) we consider $\mathcal A^{\vee}=\mathcal O(1)^d \oplus \mathcal O^l$ and $\delta=0$. We have the following possibilities:
\begin{enumerate}
\item $d=5$, $l=0$, $\mathcal H^{\vee}=\mathcal O$ and  $\mathcal E^{\vee}=\mathcal O(1)^2 \oplus \Omega^1(3)$, i.e. $\mathcal E$ has resolution 
 $$0\leftarrow \mathcal E\leftarrow \mathcal O(-1)^2 \oplus  \mathcal O(-2)^3 \leftarrow  \mathcal O(-3)\leftarrow 0 .$$
 
 \item $d=6$, $l=0$, $\mathcal H^{\vee}= \mathcal O^2$, $\mathcal E^{\vee}= \Omega^1(3)^2$ i.e. $\mathcal E$ has resolution 
$$0\leftarrow \mathcal E\leftarrow \mathcal O(-2)^6 \leftarrow  \mathcal O(-3)^2 \leftarrow 0 .$$

 \item $d=7$, $l=0$, $\mathcal H^{\vee}= \mathcal O^3$, $0\leftarrow \mathcal E\leftarrow \mathcal O(-2)^3 \oplus \mathcal O(-3)^3\leftarrow \mathcal O(-4)^2\leftarrow 0 .$
  \item $d=8$, $l=6$, $\mathcal H^{\vee}= \Omega^1(1)^2 \oplus \Omega^1(2)^2 \oplus \mathcal O^2$, i.e. ${\mathcal E}$ has resolution
$$0\leftarrow \mathcal E\leftarrow \mathcal O(-3)^4 \oplus \mathcal O(-2) ^2 \leftarrow  \mathcal O(-4)^2 \leftarrow 0 .$$
 \item $d=9$, $l=9$, $\mathcal H^{\vee}= \Omega^1(1)^3 \oplus \Omega^1(2)^3 \oplus \mathcal O(1)\oplus \mathcal O(-1)$, and 
  $$0\leftarrow \mathcal E\leftarrow \mathcal O(-3)^6 \leftarrow \mathcal O(-4)\oplus \mathcal O(-5)\leftarrow 0 .$$
\end{enumerate}

\item In case (\ref{eq even theta}) we have $\mathcal A^{\vee}= \mathcal O^{2d}$, $\delta=1$. Moreover, the following hold: 
\begin{enumerate}
\item $d=4$, $\mathcal H^{\vee}=\Omega^1(1)^2$, $\mathcal E=\mathcal O(-1)^2 \oplus \mathcal O^2$.

\item $d=5$, $\mathcal H^{\vee}=\Omega^1(1)^2\oplus \mathcal O\oplus \mathcal O(-1)$, and $\mathcal E$ has resolution
$$0\leftarrow \mathcal E\leftarrow \mathcal O(-1)^5\leftarrow \mathcal O(-2)\leftarrow 0.$$

\item $d=6$, $\mathcal H^{\vee}=\Omega^1(1)^4$, $\mathcal E=\mathcal O(-1)^4$.

\item $d=7, \mathcal H^{\vee}=\Omega^1(1) ^4 \oplus \mathcal O\oplus \mathcal O(-1)$, and $\mathcal E$ has resolution
$$0\leftarrow \mathcal E\leftarrow \mathcal O(-1)\oplus  \mathcal O(-2)^5 \leftarrow \mathcal O(-3)^2 \leftarrow 0.$$
 \end{enumerate}

\item In case (\ref{eq odd theta}) we have $\mathcal A^{\vee}= \mathcal O^{2d-3}\oplus \mathcal O(1)$, $\delta=1$, and 

\begin{enumerate}

\item $d=5$,  $\mathcal H^{\vee}=2\Omega^1(1)$, $\mathcal E=3\mathcal O(-1)\oplus \mathcal O$. 

\item $d=6$,  $\mathcal H^{\vee}= \Omega^1(1)^2\oplus \mathcal O\oplus \mathcal O(-1)$, and $\mathcal E$ has resolution
$$0\leftarrow \mathcal E\leftarrow \mathcal O(-1)^3\oplus \mathcal O(-2)^2\leftarrow \mathcal O(-3)\leftarrow 0. $$

 \end{enumerate}

\end{enumerate}
In each of the cases mentioned before, we would like to find a symmetric map $$F\colon \mathcal E(-\delta)\to \mathcal E^{\vee}$$ defining a weak quadric surface bundle similar to the quadric bundle with diagonal form $\langle x,y,xy, F(x,y,z)\rangle$ on $\mathbb{C}(\mathbb{P}^2)$ from \cite{HPT} (see \cite{T} for the general theory of quadrics). In the cases where $\mathcal E$ is a split bundle this is done in \cite[Proof of Cor. 2]{Sch1}. To deal with the remaining cases, we work with the free resolutions  
$$0\leftarrow \mathcal E \xleftarrow{\kappa} \mathcal M\xleftarrow{\rho} \mathcal N \leftarrow 0.$$ We show that having such a resolution, a quadric bundle $Q_{F}$ is represented over $\mathbb{C}(\mathbb P^2)$ by a suitable $4\times 4$ ``corner`` sub-matrix of the matrix $M_F$ of polynomials representing $\kappa^{T}\circ \phi \circ \kappa$. Indeed, we can find a map $\rho\colon \mathcal N\to \mathcal M$ of everywhere maximal rank, which defines $\mathcal E$, with a distinguished square block $J$ of size $\rk \mathcal N$ and  generic maximal rank, as well as a symmetric map $M\colon \mathcal M(-\delta)\to \mathcal M^{\vee}$ such that $M \circ \rho=0$ and $M$ has a diagonal block $M^J$  similar to $\langle x,y,xy, F(x,y,1)\rangle $, where $$F(x,y,z)=x^2+y^2+z^2-2(xy+xz+yz).$$ 
In each case we use the resolutions of $\mathcal E$ presented above. Note that in each case  $\rk \mathcal N\leq 2$ and $ 5\leq\rk \mathcal M\leq 6$ with  $\rk \mathcal M-\rk \mathcal N=4$.
To be explicit, we consider maps $\rho$ given by matrices of the following shape:
$$
\left(\begin{array}{c}
z^{a_1}\\
y^{a_2}\\
0\\
0\\\hline 
x^{b_1}\\
\end{array}\right) \quad \text{or } \quad \left(\begin{array}{cc}
z^{a_1}&0\\
y^{a_2}&0\\
0& y^{a_3}\\
0& x^{a_4}\\\hline 
x^{b_1}& 0\\
0& z^{b_2}
\end{array}\right),$$
with $a_i>0$ and the lower block representing $J$. $M$ is a matrix with diagonal block $ M^J$ of the shapes below:
$$M=\left(\begin{array}{cccc|cc} 
f_1&0&0&0&g_{11}&g_{21}\\
0&f_2&0&0&g_{12}&g_{22}\\
0&0&f_3&0&g_{13}&g_{23}\\
0&0&0&f_4&g_{14}&g_{24}\\\hline
g_{11}&g_{12}&g_{13}&g_{14}&q_{11}&q_{21}\\
g_{21}&g_{22}&g_{23}&g_{24}&q_{21}&q_{22}\\

\end{array}\right) \quad \text{or } M=\left(\begin{array}{cccc|c} 
f_1&0&0&0&g_{11}\\
0&f_2&0&0&g_{12}\\
0&0&f_3&0&g_{13}\\
0&0&0&f_4&g_{14}\\\hline
g_{11}&g_{12}&g_{13}&g_{14}&q_{11}
\end{array}\right)$$

 Then $M\rho=0$ gives the conditions:
$$
\left\{
\begin{array}{l}
f_1z^{a_1}=-g_{11}x^{b_1},  \quad f_2y^{a_2}=-g_{12}x^{b_1}, \quad f_3y^{a_3}=-g_{23}z^{b_2}, \quad f_4x^{a_4}=-g_{24}z^{b_2}; \\ \\
g_{11}z^{a_1}+g_{12}y^{a_2}=x^{b_1} q_{11}, \quad g_{23}y^{a_3}+g_{24}x^{a_4}=-z^{b_2} q_{22}, \quad g_{21}=g_{22}=g_{13}=g_{14}=q_{21}=0. \\
\end{array}
\right.
$$
As a consequence, there exist polynomials $\{g_{ij}\}$, $\{q_{ij}\}$  for which $M\rho=0$, provided that  $x^{2b_1}|f_i$ for $i=1,2$ and $z^{2b_2} | f_j$ for $j=3,4$. Similarly, for $5\times 5$ matrices the condition $x^{2b_1}|f_i$ for $i=1,2$ is sufficient. Hence we need to check that these conditions can be satisfied by some diagonal matrix $\langle f_1,\dots, f_4\rangle $ similar to $\langle x,y,xy,F(x,y,1)\rangle$. We do it in each case and present below the results for chosen examples.

For example, for an even half-period of degree $10$ we are looking for $(f_1,\dots, f_4) $ of degrees $(4,4,2,2)$ such that $x^2|f_i$ for $i=1,2$. This is fulfilled by the diagonal matrix $$\langle f_1,\dots, f_4\rangle  =\langle F(x,y,z),xy,x,y\rangle .$$ 
For half-periods of degree 18 we have degrees $(6,6,6,6)$ such that $x^2|f_1,f_2$, $z^4|f_3,f_4$ and a solution is given by $\langle xy,F(x,y,1),x, y\rangle $. For an odd theta characteristic of degree $12$ we have degrees $(5,3,3,3)$ such that $x^2|f_1,f_2$ and a solution is given by $\langle x,y,F(x,y,1), xy\rangle .$ For an even theta characteristic of degree $14$ we have degrees $(5,5,5,3)$ such that $x^2|f_1,f_2$, $z^2|f_3,f_4$ and a solution is given by $\langle x,xy,F(x,y,1), y\rangle .$

 \end{proof}

% Recall that to prove stable irrationality of quadric surface bundles associated with very general  symmetric maps $ \mathcal T(-\delta) \to \mathcal T ^{\vee}$, it is enough to construct a family of quadric surface bundles birational to them, which degenerates to a birational model of $Q_F$, and conclude via \cite[thm 9]{Sch1}. In fact, as observed before (cf. Proposition \ref{qw})  it is enough to deal with quadric bundles arising in Corollaries \ref{cor half period}, \ref{cor even theta}, \ref{cor odd theta}. 

Consider a pair $(\mathcal A,\delta)$, where $\mathcal A$ is a split bundle satisfying the assumption of one of: Corollary \ref{cor half period}, Corollary \ref{cor even theta}, Corollary \ref{cor odd theta}. We use Lemma  \ref{extension of sch}  to produce a sub-bundle $\mathcal E\subset \mathcal A$ of rank $4$ and a quadric bundle $Q_0$ over $\mathbb P^2$ defined by a symmetric map $$\phi\colon \mathcal E (-\delta)\oplus \mathcal H (-\delta) \to \mathcal E^{\vee} \oplus \mathcal H^{\vee}$$ with
$\phi((e,h))=(F(e),D(h))$, where $F\colon \mathcal E(-\delta)\to  \mathcal E^{\vee}$ satisfies the conditions in \cite[Theorem 9]{Sch1} and $D\colon  \mathcal H(-\delta)\to  \mathcal H^{\vee}$ is a nowhere degenerate map, as in Lemma \ref{extension of sch}.
In that case, the quadric surface bundle $Q_F\subset \mathbb {P}(\mathcal E^{\vee})$ associated with $F$ is irrational and has the property that the very general element in any irreducible flat family of bundles degenerating to it is irrational \cite[Theorem 9]{Sch1}.  Note that $Q_0$ and $Q_F$ admit the same discriminant and associated cokernel but have different rank.

Let furthermore $\mathcal R$ be the corresponding moduli space of pairs $(C,\mathcal L)$ (half-periods, even or odd theta characteristics). By \cite{B}, we have a dominant map $H^0(S^2(\mathcal A(\delta)))\to \mathcal R$. Consider a one parameter family $q_t\in H^0(\operatorname{Sym}^2(\mathcal A(\delta)))$ with $q_0$ mapping to $F$ via the restriction map $H^0(\operatorname{Sym}^2 \mathcal A(\delta))\to H^0(\operatorname{Sym}^2 \mathcal E(\delta))$. Applying Lemma \ref{degeneration from extension}, we obtain a one-parameter family of quadric bundles $\{Q_t\}_{t\in {\mathbb C}}$ degenerating in $0$ to $Q_0$.

\begin{lemma} \label{Qt} There exists a family of isotropic sub-bundles $\mathcal{U}_t\subset \mathcal A$ for $t$ in an \'etale neighborhood of $0$ which are isotropic with respect to $Q_t$ and such that the corresponding quadric reduction $Q_t^{\mathcal U}$ has rank $4$. \end{lemma}
\begin{proof} Indeed, we saw in the proof of Corollaries \ref{cor half period}, \ref{cor even theta}, \ref{cor odd theta}, that for each fixed $t$ we have a family $\mathcal I_t$ of isotropic subbundles $\mathcal U_{t,\varphi}$ obtained as images of maximal rank morphisms $\varphi_t\colon \mathcal O(-1)^{l_1}\oplus \mathcal O(-2)^{l_2} \to \mathcal A$ for appropriate $l_1,l_2\in \mathbb Z$. Hence the family $\mathcal I_t$ is parametrized by a subset $\mathcal M_t\subset  \Hom(\mathcal O(-1)^{l_1}\oplus \mathcal O(-2)^{l_2},\mathcal A)$ which is a subset of the space of matrices with polynomial entries whose coefficients satisfy a system of quadratic equations listed explicitly in Corollaries \ref{cor half period}, \ref{cor even theta}, \ref{cor odd theta}. These equations linearly involve the parameter $t$. As a consequence, we have a scheme $$\mathcal M\subset \Hom(\mathcal O(-1)^{l_1}\oplus \mathcal O(-2)^{l_2},\mathcal A)\times \mathbb{C}^*$$ which admits a surjective projection $\pi\colon \mathcal M\to \mathbb{C}^*$. 
Any local section of $\pi$  in a neighborhood of $0$ leads to a family of isotropic sub-bundles $\{\mathcal U_t\}_{t\in \mathbb C}$. Thus the Lemma follows, as an \'etale local section of $\pi$ always exists. \end{proof}

\begin{prop}\label{Q0'}
Let $\{Q_t^{\mathcal U}\}_t$ be the family of quadric reductions of $Q_t$ with respect to $U_t$ for $0<|t|<<1$. Then there exists a quadric reduction $Q'_0$ of $Q_0$ over some open subset $V\subset \mathbb{P}^2$, such that $Q'_0$ is birational to $Q_F$ and appears as a flat limit of $Q_t^U|_V$.
\end{prop}
\begin{proof}
Note that the flat limit of the family of subbundles $ \mathbb P (U_t^{\vee})$ is a subvariety of $\mathbb{P}(\mathcal A_0)$ which is a weak sub-bundle (i.e. a sub-bundle when restricted to an open subset of $\mathbb{P}^2$) of rank equal to the dimension of $U_t$. Let $ V\subset \mathbb{P}^2$ be the open subset  where $U_0$ is a regular isotropic sub-bundle.  Consider the reduction $Q'_0$ of $Q_0|_V$ with respect to $U_0$ over $V$. By Proposition \ref{prop quadric reduction brauer class}, $Q'_0$ and $Q_0$ define the same Brauer class on their common discriminant double cover. On the other hand, by Theorem \ref{general IOOV} we know that $Q_0$ and $Q_F$ have a common discriminant double cover and define the same Brauer class. Finally, by  
Corollary  \ref{birational quadric bundle}, we conclude that $Q'_0$ is birational to $Q_F$ over $V$.
\end{proof}

Now, Corollary  \ref{cor irrationality} is a straightforward consequence of the discussion above.
\begin{proof}[Proof of Corollary  \ref{cor irrationality} ]
The quadric surface bundle $Q'_0|_V$ in Proposition \ref{Q0'} is birational to $Q_F$, hence it satisfies the assumptions of \cite[Theorem~9]{Sch1}. In particular, $Q^{\mathcal U}_t$ is not stably rational for very general $t$. Since in Lemma \ref{Qt}, we can choose any one-parameter family $\{Q_t\}_{t\in \mathbb C}$, and the union of them dominates the moduli space $\mathcal R$, the quadric surface bundles with very general (in $\mathcal R$) cokernel are also not stably rational. Moreover, the forgetful map from $\mathcal R$ to the moduli space of curves is finite; hence a quadric bundle with very general discriminant is also mapped to a very general element in $\mathcal R$ and hence it is not stably rational. %This concerns also $\widetilde Q$ from the beginning of the chapter.
\end{proof}

\begin{rem}\label{???1} Corollary \ref{cor irrationality} does not take into account the case of odd theta characteristics of degree $8$. The stable irrationality of these quadric surface bundles remains an open problem.
\end{rem}

\begin{rem}\label{???2}
We cannot use irrationality of some quadric bundle to prove irrationality of very general elements in their family defined by $(C,\cc)$. Indeed, a family of curves equipped with quasi half-periods does not need to admit a family of quadric surface bundles with these as associated cokernels. As an example, denote by $\mathcal Q_{\mathcal T}$ the family of quadric bundles which are given by symmetric maps $\varphi\colon \mathcal T(-1)\to \mathcal T^{\vee}$, where $\mathcal{T}^{\vee}=\Omega^1(2)\oplus \mathcal O(1)\oplus \mathcal O$. This is a family of quadric bundles that represent rank $4$ symmetric resolutions of general elements $(C,\cc)$ in the moduli space $\mathcal R$ of odd theta characteristics supported on curves of degree 8. On the other hand, there exists a family $\mathcal Q_{\mathcal S}$ of quadric bundles (studied in \cite{Sch1}) which corresponds to symmetric maps $\theta\colon \mathcal O(-1)^3 \oplus \mathcal O(-3) \to \mathcal O^3 \oplus \mathcal O(2)$ such that the very general element is not stably rational. Moreover, the cokernels are also odd theta characteristics on curves of degree $8$. Although every sheaf $\coker \theta$ appears as a degeneration of a family of sheaves of the form $\coker \varphi$, these degenerations do not extend to flat families of quadric bundles. In fact, a simple dimension count tells us that the space of deformations of quadric bundles in $\mathcal Q_{\mathcal S}$ has dimension higher than $61$, whereas the space of deformations of quadric bundles in  $\mathcal Q_{\mathcal T}$ has dimension smaller than $61$. One also checks that flat limits of quadric bundles from $\mathcal Q_{\mathcal T}$, which belong to $\mathcal Q_S$, are all either rational (they admit an isotropic section) or generically singular. As an example, we can consider the quadric reduction given by
$$0\to \Omega^1(1)^2\oplus \Omega^1\oplus \Omega^1(2)\to \oo^5\oplus \oo(1) \oplus (\oo\oplus \oo(-1))^3 \to \oo^3 \oplus \oo(2)\to 0.$$
 
\end{rem}
\begin{rem}The bounds for the degrees in Corollary \ref{cor irrationality} are related to the problem of constructing explicit dominant families of quadric surface bundles.  It was pointed {\color{black} out} to us by B. Hassett, that dominant families in the case of quadric bundles with higher degree discriminant could possibly be found using general techniques developed in \cite[\S 4.3]{HKT}.
\end{rem}

\section{Nodal Gushel--Mukai fourfolds -- Proof of Corollary \ref{corr}}\label{1.4}

Additionally to the models described in Corollary \ref{corr} there is also another family of varieties that are birational to general Verra fourfolds. These are one-nodal Gushel--Mukai fourfolds, characterized in the following way.

\begin{lem} \label{genNGM}There are two types of one-nodal Gushel--Mukai fourfolds:
\begin{enumerate}
    \item The intersections $W\cap C(Q_0)$, where $C(Q_0)$ is a quadric cone with vertex $P$ over a smooth quadric $Q_0\subset \PP^7$ and $W=G(2,5)\cap H$ is a smooth Fano fivefold, which is the intersection of the Grassmannian $G(2,5)$ in its Pl\"ucker embedding with a hyperplane $H$.
    
    \item The intersections of a quadric cone $C(Q_0)$ with vertex $P$  with $W'=C(G(2,5)\cap H_1 \cap H_2$ being a cone over a smooth linear section of the Grassmannian $G(2,5)$.

\end{enumerate}
Moreover, fourfolds of the second type are degenerations of fourfolds of the first type. 
\end{lem}
\begin{proof}
{\color{black}
For the proof we can mimic the proof of the analogous statement in the case of nodal Gushel-Mukai threefolds given in \cite[Lemma 4.1]{DIM0}.}
 % {\color{blue} OLD Recall that a Gushel--Mukai fourfold is the intersection of a cone $C(G)$ over the Grassmannian $G=G(2,5)$ with a codimension two linear space $L$ and a quadric. Then $L\cap C(G)$ is either smooth in which case we are in the situation above or $L\cap C(G)$ passes through the vertex of the cone and is singular only at this vertex. In all other situations, the singular locus of $L\cap C(G)$ would be positive dimensional and hence the Gushel--Mukai fourfold couldn't be one-nodal. It remains to point out that Gushel--Mukai fourfolds for which $L$ passes through the vertex  can be deformed to the case in which it doesn't, still keeping the singularity.}
\end{proof}
We can now formulate the relation between the families of Verra fourfolds and one-nodal Gushel--Mukai fourfolds. 
\begin{cor}\label{verbirNGM}
 A general Verra fourfold is birational to some one-nodal Gushel--Mukai fourfold. Moreover, a general one-nodal Gushel--Mukai fourfold is birational to some Verra fourfold.  
\end{cor}
Let $X$ be a general one-nodal Gushel--Mukai fourfold as in Lemma \ref{genNGM}(1).
As a consequence of the construction, the projection of $X$ from the node $P$ is contained in $Q_0$. This projection $X_0$ is the complete intersection of three quadrics containing a quadric threefold $Q\subset \PP^4_Q\subset \PP^7$, where $\PP^4_Q$ is a $4$-dimensional projective space containing $Q$. Indeed, by the explicit Pfaffian equations of $W$, we deduce that the image $W_0\subset \PP^7$ of the projection of $W$ from $P$ is the intersection of two quadrics $Q_1,Q_2$ containing $\PP^4_Q\subset \PP^7$. If we blow up $\PP^7$ in $\PP^4_Q$, we obtain a $\PP^5$-bundle $\widetilde{\PP^7}\to \PP^2$ containing a $\PP^3$-bundle $\widetilde{W_0}$. Denote by $R_X\subset \widetilde{W_0}$ the proper transform of $X_0$ through that blow up.

\begin{lem}\label{5.3} We have $\widetilde{\PP^7}=\PP(\oo_{\PP^2}^5 \oplus \oo_{\PP^2}(1))$ and $\widetilde{W_0}=\PP(\mathcal{G}^{\vee})$
where $$0\to \oo_{\PP^2}(-1)^2\to \oo_{\PP^2}^5 \oplus \oo_{\PP^2}(1)\to \mathcal{G}^{\vee}\to 0.$$

\end{lem}  
\begin{proof}
The exact sequence is due to the embedding $\PP(\mathcal{G}^{\vee})\subset
\PP(\oo_{\PP^2}^5 \oplus \oo_{\PP^2}(1))$. The map from $\oo_{\PP^2}(-1)^2$ corresponds to the two quadrics defining $W_0$. 
\end{proof}
\begin{rem}
One can see that for a general $\mathcal{G}$ as above, the bundle $\mathcal{G}^{\vee}$ admits an injective resolution $0\to \mathcal{G}\to \oo_{\PP^2}(1)^4\oplus \oo_{\PP^2}(2)\to \oo_{\PP^2}(3)\to 0$. This can be proved by means of a resolution of $\mathcal{G}^{\vee}$ and using a computer algebra program.
\end{rem}
Note that the resolution in Lemma \ref{5.3} is of the same shape as that in Example \ref{po}(2), hence we can use Corollary \ref{deg6} once we prove that corresponding quadric bundle is general enough. 

Let $\oo_{\PP(\mathcal G^{\vee})}(1)$ be the sheaf inducing the map $\PP(\mathcal{G}^{\vee})\to W_0\subset \PP^7$. Then there exists a non-zero global section $G\in H^0(\oo_{\PP(\mathcal{G}^{\vee})}(2)) \simeq H^0(\Sym^2(\mathcal G^{\vee}))$, which induces a family of quadratic forms $q\colon \oo_{\PP^2}\to \Sym^2(\mathcal{G}^{\vee})$, whose zero locus is $R_X$. This induces a symmetric map of sheaves on $\PP^2$ whose cokernel is denoted by $\cc$, namely: 
$$0\to \mathcal{G}\xrightarrow{q} \mathcal{G}^{\vee}\to \cc\to 0.$$ Taking the determinant  of $q_G$ gives a canonical isomorphism (see Lemma \ref{quasi half period})
$$\det(q_G)\colon \cc\otimes \det(\mathcal{G}^{\vee})\to (\cc\otimes \det(\mathcal{G}^{\vee}))^{\vee}.$$ Hence $\cc(-3)$ is a half-period on the discriminant sextic curve $C\subset \PP^2$.

\begin{lem}
The half-period $\cc(-3)$
is general, i.e.~$h^0(\cc(-2))=0$.
\end{lem}
\begin{proof}
Apply the long cohomology exact sequence coming from the exact sequence defining  $\mathcal G$.
\end{proof}

Since $\cc(-3)$ is general, we deduce a free resolution from \cite[Proposition 3.1]{B}:
$$0\to \oo_{\PP^2}(-1)^3\xrightarrow{N} \oo_{\PP^2}(1)^3\to \cc\to 0,$$
where $N$ is a symmetric matrix of quadrics.
If we expand $N$ by the identity on an additional component $\oo_{\PP^2}$, we obtain the following non-minimal resolution: $$0\to \oo_{\PP^2}(-1)^3\oplus \oo_{\PP^2}\xrightarrow{N'} \oo_{\PP^2}(1)^3\oplus \oo_{\PP^2}\to \cc\to 0.$$

Notice that $N'$ gives a section of  $H^0(\Sym^2(\oo_{\PP^2}(1)^3\oplus \oo_{\PP^2})) \simeq H^0(\oo_{\PP(\oo_{\PP^2}(1) \oplus \oo_{\PP^2})}(2))$ and is associated to the case (4) in Example \ref{po}.

\begin{lem}\label{Verra} The image of $\PP(\oo_{\PP^2}(1)^3\oplus \oo_{\PP^2})$ via the linear system
$|\oo_{\PP(\oo_{\PP^2}(1)^3\oplus \oo_{\PP^2})}(1)|$ is a cone $C(\PP^2\times \PP^2)\subset \PP^9$ and $N'$ defines a Verra fourfold $V_X$. 
\end{lem}
\begin{proof} Observe that $\PP(\oo_{\PP^2}(1)^3)$ is isomorphic to $\PP(\oo_{\PP^2}^3)$ so to $\PP^2\times \PP^2$  and its tautological divisor defines  the Segre embedding $\PP^2\times \PP^2\subset \PP^8$.
Hence the tautological bundle on  $\PP(\oo_{\PP^2}\oplus \oo_{\PP^2}(1)^3)$ maps the considered projective bundle to the cone $C(\PP^2\times \PP^2)\subset \PP^9$ over the image of $\PP(\oo_{\PP^2}(1)^3)$ in $\PP^8$. Now $N'$ is associated to a quadric section of the image hence to a Verra fourfold. 
\end{proof}
Summing up, to a general nodal Gushel--Mukai fourfold $X$ we associated a Verra fourfold $V_X$.
\begin{lemma}\label{VXRX}
 Let $X$ be a general one-nodal Gushel--Mukai fourfold. Then the quadric bundles $R_X$ and $V_X$ {\color{black} are isomorphic at the generic point of} $\mathbb{P}^2$. In particular,  the Gushel--Mukai fourfold $X$ and the Verra fourfold $V_X$ are birational.
\end{lemma}
\begin{proof} The quadric bundles $R_X$ and $V_X$ induce the same half-period $\cc$ on the discriminant curve $C$, so we can apply Theorem \ref{main}.
 \end{proof}
 
 \begin{proof}[Proof of Corollary \ref{verbirNGM}]
 The first assertion follows directly from Lemma \ref{VXRX}. To prove the second assertion start with a general Verra fourfold $V$ and using Corollary \ref{deg6}
find a birational model of $V$ that corresponds to a quadric bundle $R$ given in $\widetilde{\mathbb P^7}$ as the intersection of three divisors as in Corollary \ref{deg6}(2). The image of the projection of $R$ to $\mathbb P^7$ is then  a complete intersection $X_0$ of three quadrics containing a three-dimensional quadric. It follows that $X_0$ can be seen as a the intersection of:  a complete intersection $W_0$ of two quadrics containing a $\mathbb{P}^4$ in $\mathbb{P}^7$, and another quadric $Q_0$. By standard unprojection \cite[\S 2.5.1 ]{PapadakisThesis} the variety $W_0$ is the projection of a hyperplane section of the Grassmannian $G(2,5)$ (in its Pl\"ucker embedding) from a point on it. Hence $W_0\cap Q_0$ is birational to its preimage via the projection which is a nodal Gushel--Mukai fourfold.
 \end{proof}
 
\section*{Acknowledgements}
G.K.~is supported by the project Narodowe Centrum Nauki 2018/30/E/ST1/00530. M.K.~is supported by the project Narodowe Centrum Nauki 2018/31/B/ST1/02857.  G.B. thanks the Jagiellonian University and IMPAN for warm hospitality during his stay in Cracow in December 2018. We would like to thank S.~Kuznetsov for his interest in our result and many valuable comments.
We also thank B.~Hassett, R.~Laterveer, J.~Ottem, A.~Kresch,  B.~van Geemen, and J.~Weyman for useful comments, as well as the anonymous referees for important remarks.

\newpage

 \centerline{\sc APPENDIX A.} 
 {\color{black}The following script can also be found in \cite{M2GH}.}
\begin{verbatim}
CheckDom = (r,l)-> 
(k=4-l;
n=3;
S=ZZ[x_1..x_(n*r*(r+k))];
DS=ZZ[y_1..y_(r*l)];
R=ZZ[t_1..t_n];
T=S**DS**R;
for j from 1 to 5 do
(
if l==0 then K=0 else K=genericMatrix(DS,l,r);
M=random(R^r, R^{l:-2},Height=>2);
ZD=random(R^l, R^{l:-2},Height=>2);
Z=ZD+transpose ZD;
if l==0 then B=0 else B= (transpose((map(T,DS)) K))*((map(T,R)) Z)*((map(T,DS)) K);
if l==0 then F=0 else F=( ((map(T,R)) M))*((map(T,DS)) K);
if l==0 then G=0 else G=F+transpose F;
N1=genericMatrix(S,n*r,r+k);
V=(vars R)**diagonalMatrix(for i from 0 to r-1 list 1);
N=((map(T,R))V)*((map(T,S))N1);
etaKN=(N*(transpose N))+G+B;
IKN=mingens  ideal flatten etaKN;
DKN1=mingens ideal flatten diff(matrix{{t_1..t_n}},diff(matrix{{t_1..t_n}},IKN));
DKN=(map(S**DS,T)) DKN1;
JDKN=jacobian DKN;
JDKNx=(map(ZZ,ring JDKN, random(ZZ^1,ZZ^(n*r*(r+k)+r*l))))JDKN;
ds=numgens source JDKNx;
dt=numgens target JDKNx;
LM=for i from 1 to 5 list det(random((ring JDKNx)^ds,(ring JDKNx)^dt, Height=>2)*JDKNx;
g_j=gcd(LM));
);
return factor gcd(g_1,g_2,g_3,g_4,g_5))
\end{verbatim}

\end{document}